\documentclass[11pt]{amsart}

\usepackage{graphicx,amssymb} 
\usepackage{epsfig}
\usepackage{textcomp}

\usepackage{epstopdf}

\usepackage{fancyhdr}

\usepackage{textcomp}

\begin{document}



\title{Fractal analysis of Hopf bifurcation at infinity}

\author{\fontsize{11pt}{13pt} GORAN RADUNOVI\' C, DARKO \v ZUBRINI\'C and VESNA \v ZUPANOVI\'C}

\address{Department of Applied Mathematics, Faculty of Electrical Engineering and Computing\\ University of Zagreb,
10000 Zagreb, Croatia}



\begin{abstract}
Using geometric inversion with respect to the origin we extend the definition of box dimension to the case of unbounded subsets of Euclidean spaces. Alternative but equivalent definition is provided using stereographic projection on the Riemann sphere. We study its basic properties, and apply it to the study of the Hopf-Takens bifurcation at infinity.
\end{abstract}



\newtheorem{theorem}{Theorem}[section]
\newtheorem{cor}[theorem]{Corollary}
\newtheorem{prop}[theorem]{Proposition}
\newtheorem{lemma}[theorem]{Lemma}

\theoremstyle{remark}
\newtheorem{remark}[theorem]{Remark}
\newtheorem{defn}[theorem]{Definition}
\newtheorem{example}[theorem]{Example}
\newtheorem{problem}[theorem]{Problem}
\newtheorem{exercise}[theorem]{Exercise}

\newcommand{\re}{\operatorname{Re}}
\newcommand{\im}{\operatorname{Im}}

\font\csc=cmcsc10

\def\esssup{\mathop{\rm ess\,sup}}
\def\essinf{\mathop{\rm ess\,inf}}
\def\wo#1#2#3{W^{#1,#2}_0(#3)}
\def\w#1#2#3{W^{#1,#2}(#3)}
\def\wloc#1#2#3{W_{\scriptstyle loc}^{#1,#2}(#3)}
\def\osc{\mathop{\rm osc}}
\def\var{\mathop{\rm Var}}
\def\supp{\mathop{\rm supp}}
\def\Cap{{\rm Cap}}
\def\norma#1#2{\|#1\|_{#2}}

\def\C{\Gamma}

\let\text=\mbox

\catcode`\@=11
\let\ced=\c
\def\a{\alpha}
\def\b{\beta}
\def\c{\gamma}
\def\d{\delta}
\def\g{\lambda}
\def\o{\omega}
\def\q{\quad}
\def\n{\nabla}
\def\s{\sigma}
\def\div{\mathop{\rm div}}
\def\sing{{\rm Sing}\,}
\def\singg{{\rm Sing}_\ty\,}

\def\A{{\cal A}}
\def\F{{\cal F}}
\def\H{{\cal H}}
\def\W{{\bf W}}
\def\M{{\cal M}}
\def\N{{\cal N}}
\def\S{{\cal S}}

\def\eR{{\bf R}}
\def\eN{{\bf N}}
\def\Ze{{\bf Z}}
\def\Qe{{\bf Q}}
\def\Ce{{\bf C}}

\def\ty{\infty}
\def\e{\varepsilon}
\def\f{\varphi}
\def\:{{\penalty10000\hbox{\kern1mm\rm:\kern1mm}\penalty10000}}
\def\ov#1{\overline{#1}}
\def\D{\Delta}
\def\O{\Omega}
\def\pa{\partial}

\def\st{\subset}
\def\stq{\subseteq}
\def\pd#1#2{\frac{\pa#1}{\pa#2}}
\def\sgn{{\rm sgn}\,}
\def\sp#1#2{\langle#1,#2\rangle}

\newcount\br@j
\br@j=0
\def\q{\quad}
\def\gg #1#2{\hat G_{#1}#2(x)}
\def\inty{\int_0^{\ty}}
\def\od#1#2{\frac{d#1}{d#2}}

\def\bg{\begin}
\def\eq{equation}
\def\bgeq{\bg{\eq}}
\def\endeq{\end{\eq}}
\def\bgeqnn{\bg{eqnarray*}}
\def\endeqnn{\end{eqnarray*}}
\def\bgeqn{\bg{eqnarray}}
\def\endeqn{\end{eqnarray}}

\def\bgeqq#1#2{\bgeqn\label{#1} #2\left\{\begin{array}{ll}}
\def\endeqq{\end{array}\right.\endeqn}

\def\abstract{\bgroup\leftskip=2\parindent\rightskip=2\parindent
        \noindent{\bf Abstract.\enspace}}
\def\endabstract{\par\egroup}

\def\udesno#1{\unskip\nobreak\hfil\penalty50\hskip1em\hbox{}
             \nobreak\hfil{#1\unskip\ignorespaces}
                 \parfillskip=\z@ \finalhyphendemerits=\z@\par
                 \parfillskip=0pt plus 1fil}
\catcode`\@=11

\def\cal{\mathcal}
\def\eR{\mathbb{R}}
\def\eN{\mathbb{N}}
\def\Ze{\mathbb{Z}}
\def\Qu{\mathbb{Q}}
\def\Ce{\mathbb{C}}

\def\osd{\mathrm{osd}\,}

\def\sdim{\mbox{\rm s-dim}\,}
\def\sd{\mbox{\rm sd}\,}

\newcommand{\diag}{\operatorname{diag}}

\addtolength{\textheight}{-0cm}







\maketitle

\section{Introduction}

The main goal of dimension theory for dynamics is to measure the complexity of invariant sets and measures using fractal dimensions. A good example of this can be seen in \cite{piac} where the Hausdorff dimension of a particular case of the H\' enon attractor is estimated and compared to its box dimension. In many cases fractal dimensions can give us a better understanding of the dynamics appearing in various problems in physics, engineering, chemistry, medicine, etc. The second and the third author used the box dimension to analyse spiral trajectories of some planar vector fields in \cite{zuzu}. Among other things, they studied the Hopf bifurcation which is a well known bifurcation of $1$-parameter families of vector fields in which a limit cycle, that is, an isolated periodic orbit, is born from a singular point. The generalisation is the Hopf-Takens bifurcation which gives rise to more than one limit cycle born from a singular point. It was shown that the box dimension of spiral trajectories near singular points or nonhyperbolic (multiple) limit cycles becomes nontrivial, that is, greater than $1$ precisely at a point at which the corresponding dynamical system undergoes the bifuraction. Moreover, the box dimension can only take values from a discrete set and depends on the multiplicity of the corresponding singular point or limit cycle. This could be utilised for computing the multiplicity. In a way, this is related to the $16$th Hilbert problem of finding an upper uniform bound for the number of limit cycles in dependence on the degree of the polynomial vector field. Results about spiral trajectories of some vector fields in $\eR^3$ can be found in \cite{3}. Further studies using the asympthotic behaviour of the analytic Poincar\' e map associated to the spiral trajectories near singular points and periodic orbits can be seen in \cite{belg}.

The multiplicity of the Poincar\' e map is related to the notion of cyclicity, that is the number of limit cycles that can be born after a small perturbation of the system. Furthermore, the Poincar\' e map of a planar vector field generates a 1-dimensional discrete dynamical system. In \cite{neveda} it was shown that the box dimension of the $1$-dimensional trajectory is related to this discrete dynamical system. There the box dimension of the corresponding orbit is studied for the classical saddle-node and period doubling bifurcations. Using these results, classical theorems about these bifurcations were extended. Further extensions to 2-dimensional discrete dynamical systems and applications to continuous dynamical systems were obtained in \cite{lana}.

It is known that limit cycles can also be generated from a polycycle, which is an ordered collection of singular points (vertices) and bi-asymptotic trajectories (edges) connecting them in a specified order. Remark that an isolated singular point is a special case of a polycycle. The next simplest case is a saddle-loop, that is a polycycle with only one vertex and one edge. The Poincar\' e map near a saddle loop, although it is not analytic, shows its cyclycity (see \cite{roussarie}, \cite{liquion}). In \cite{mrz} this was investigated from the point of view of fractal geometry. The classical box dimension was not fine enough to distinguish between all the cases which could appear, so a generalisation called the critical Minkowski order has been introduced.

As limit cycles can also be born from a point or a polycycle at infinity, it makes sense to generalise the previous results to this case. It is also interesting to study the problem of Hopf-Takens bifurcation of polynomial vector fields at infinity from the fractal point of view. Related problems have been studied in \cite{Caubergh11}, \cite{Blows93} and \cite{gine}. 

In this paper we deal with vector fields possessing spiral trajectories tending to infinity. The provided visualisations clearly show that in the case of a weak focus at infinity such trajectories exhibit an
almost ``planar'' nature. We measure this phenomenon using the box dimension of trajectories. Since the trajectories tending to infinity are unbounded, we have adapted the definition of
box dimension to this case, since the usual box dimension is defined for bounded sets only. We do this using the geometric inversion, see Definition~\ref{box} below.

Let us recall the definitions of Minkowski content and box dimension.  By $|\O|$ we denote the $n$-dimensional Lebesgue measure of an open
subset $\O$ of $\eR^n$. Let $A$ be a nonempty bounded subset of $\eR^n$, $A_\e$ the $\e$-neighbourhood of $A$ in the Euclidean metric and $s\geq 0$.
The {\em upper $s$-dimensional Minkowski content} of  $A$ is defined by 
\bgeq\label{mink12}
\M^{*s}(A)=\limsup_{\e\to0}\frac{|A_\e|}{\e^{n-s}}, 
\endeq
and we define analogously the {\em lower $s$-dimensional Minkowski content} of $A$, denoted by $\M_*^s(A)$. The {\em upper box dimension} of $A$ is defined by
\bgeq\label{dim}
\ov\dim_BA=\inf\{s>0:\M^{*s}(A)=0\};
\endeq
it is easy to see that we also have
\bgeq\label{Mty}
\ov\dim_BA=\sup\{s>0:\M^{*s}(A)=\ty\}.
\endeq
The lower box dimension of $A$, denoted by $\underline \dim_BA$,  is defined analogously, using $\M_*^s(A)$ instead of  $\M^{*s}(A)$ in (\ref{dim}) (and in (\ref{Mty})). If both dimensions 
$\ov\dim_BA$ and $\underline\dim_BA$
are equal, the common value is denoted by $\dim_BA$, and is called the {\em box dimension} of $A$ (also known as Minkowski-Bouligand dimension, or limit capacity). If there exists $d\ge0$ such that $0<\M_*^d(A)\le\M^{*d}(A)<\ty$, we say that $A$ is {\em Minkowski nondegenerate},
and {\em Minkowski degenerate} otherwise.
(Note that if $A$ is nondegenerate, it then follows from (\ref{dim})--(\ref{Mty}) and their counterpart for $\M_*^s(A)$ that $\dim_BA$ exists and is equal to~$d$.)
 If $\M_*^d(A)=\M^{*d}(A)$, the common value is denoted by $\M^d(A)$, and called the {\em Minkowski content}. If moreover $\M^d(A)\in(0,\ty)$, then $A$ is said to be {\em Minkowski measurable}. For more information about these notions and their generalisations see  \cite{zuzu}, \cite{nova}, \cite{falc} and \cite{mrz}.

In the sequel we will use the following notation. If $f,g:\eR\to(0,\ty)$ are two functions such that $f(t)\to0$ and $g(t)\to0$ as $t\to t_0$ ($t_0$ can be $\ty$ as well), we write $f(t)\sim g(t)$ as $t\to t_0$
if $\lim_{t\to t_0}\frac{f(t)}{g(t)}=1$. We write $f(t)\simeq g(t)$ and say that $f$ and $g$ are {\em comparable} as $t\to t_0$ if there exist positive constants $c_{1,2}$ such that $c_1 g(t)\le f(t)\le c_2g(t)$ for all $t$
in a neighbourhood of $t_0$. A function $f:V\to\eR^n$, $V\stq\eR^n$, is said to be Lipschitzian if $|f(a)-f(b)|\simeq|a-b|$ for all $a,b\in V$.

\section{Box dimension and Minkowski content of unbounded sets}

\subsection{Definition of box dimension of unbounded sets by geometric inversion}
We start with a polynomial system
\bgeq\label{x}
\dot x=P(x)
\endeq
defined on $\eR^n$. Applying the change of variables $u=x/|x|^2$ (here $|x|$ is Euclidean norm, $|x|^2=x_1^2+\dots+x_n^2$), which is the well known geometric inversion 
of $\eR^n\setminus\{0\}$ with respect to the origin, after a short computation we arrive at the following system:
\bgeq\label{u}
\dot u=|u|^2\tilde P(u)-2u(u\cdot\tilde P(u)),
\endeq
defined on $\eR^n\setminus\{0\}$, where 
\bgeq
\tilde P(u)=P\left(\frac u{|u|^2}\right).\nonumber
\endeq
The geometric inversion is clearly involutive, so that $x=u/|u|^2$.
The right-hand side of (\ref{u}) is not necessarily a polynomial field in dependence of $u_1,\dots,u_n$, where $u=(u_1,\dots,u_n)$.
However, note that the largest exponent of $|u|^{-2}$ appearing within the component functions of $\tilde P(u)$ is equal to $k=\deg P:=\max_i\deg P_i$. Hence, 
\bgeq\label{uk}
\dot u=|u|^{2k}\left(|u|^2\tilde P(u)-2u(u\cdot\tilde P(u))\right),
\endeq
is a polynomial vector field. 

For any set $A\st\eR^n\setminus\{0\}$ we can define its geometric inverse with respect to the origin by
\bgeq
\Phi(A)=\{\Phi(x):x\in A\},\nonumber
\endeq
where $\Phi(x)=\frac x{|x|^2}$. As we have said, the mapping is involutive: $\Phi^2=id$.
If we denote the phase portrait of (\ref{x}) by $\mathcal P=\{\C_i:i\in I\}$ (the family of trajectories $\C_i$), it will be convenient to define $\Phi(\mathcal P)$ by
\bgeq\label{P-1}
\Phi(\mathcal P)=\{\Phi(\C_i):i\in I\}.
\endeq
It is clear that $\Phi(\mathcal P)$ is the phase portrait of (\ref{u}) on $\eR^n\setminus\{0\}$.
Hence, we have proved the following result.

\begin{figure}[h]
\begin{center}
\includegraphics[width=6cm]{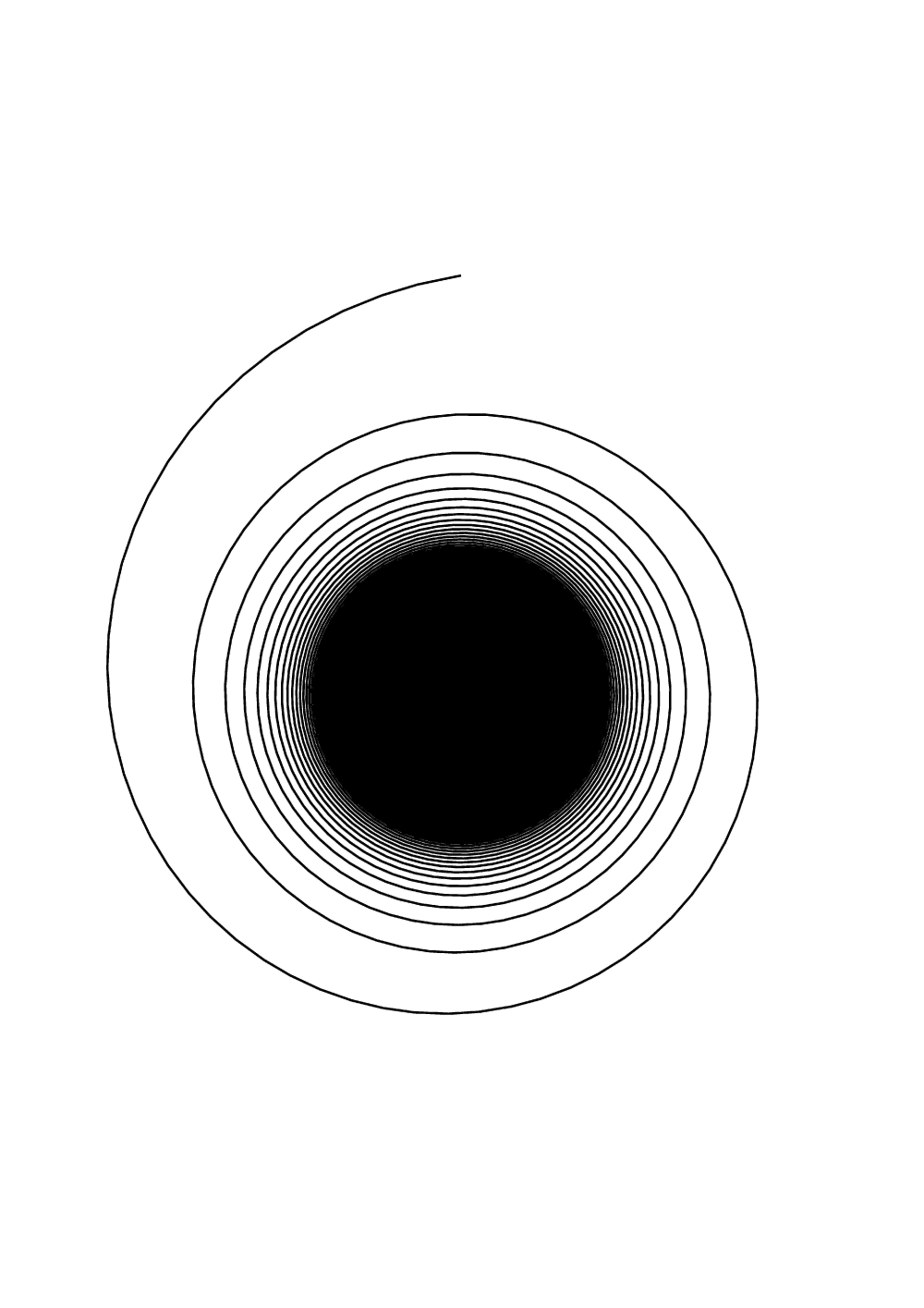}
\includegraphics[width=6cm]{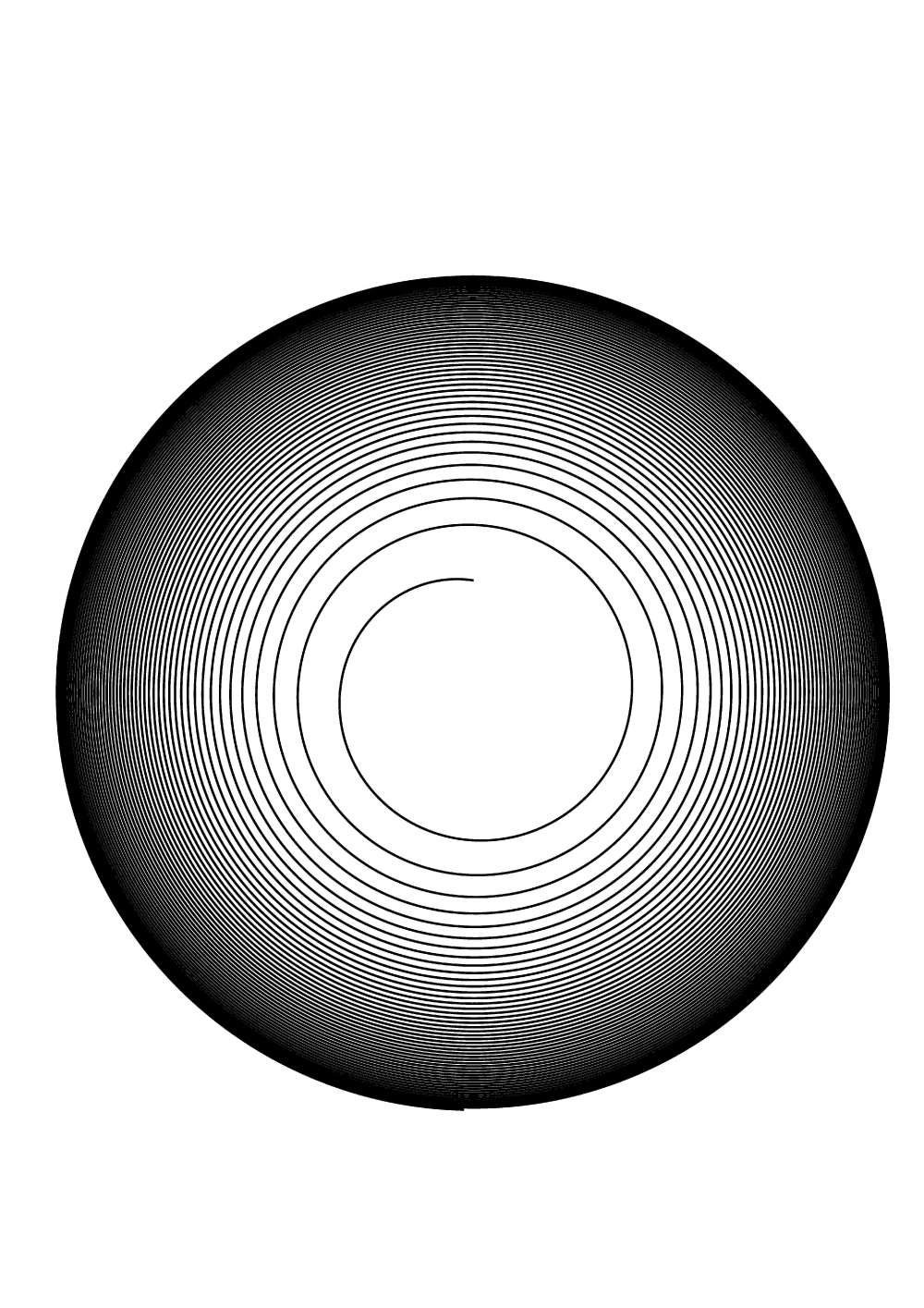}
\end{center}
\caption{The bounded spiral $r=\f^{-1/4}$ (left) and the unbounded spiral $r=\f^{1/4}$ (right) both have the same box dimension equal to $8/5$. The nucleus of the bounded spiral is at the origin whereas the nucleus of the unbounded spiral is at infinity.}
\label{unbounded_spiral}
\end{figure}

\begin{lemma}
Let a polynomial vector field $P$ in $\eR^n$ be given, and let $\mathcal P$ be the phase portrait of (\ref{x}).
Then there exists an explicit polynomial vector field in $\eR^n$, given by (\ref{uk}), such that its phase portrait is equal to $\Phi(\mathcal P)$.

In particular, if $\{C_i:i\in I\}$ is the collection of all limit cycles of a polynomial vector field, then there exists a polynomial vector field in $\eR^n$
such that $\{\Phi(C_i):i\in I\}$ is the collection of all limit cycles of the new vector field.
\end{lemma}

As we see, if $C$ is a limit cycle of a polynomial system, then its geometric inverse $\Phi(C)$ is also a limit cycle of a polynomial system.

\begin{defn}
We say that the infinite point is a weak focus of a dynamical system in $\eR^{n}$ if the origin is a weak focus of the system obtained by its geometric inversion.
\end{defn}

It will be convenient to extend the definition of the box dimension from bounded sets in $\eR^n$ to the case of unbounded sets.

\begin{defn}\label{box}
Let $A$ be an unbounded set in $\eR^n$, which is away from the origin, that is, $d(A,\{0\})=\inf\{|a|:a\in A\}>0$. 
Then clearly $\Phi(A)$ is bounded, and we define the upper box dimension of $A$ by
\bgeq
\ov\dim_BA=\ov\dim_B\Phi(A).\nonumber
\endeq
Analogously for the lower box dimension. If both the upper and lower box dimensions of $A$ coincide, we call it just the box dimension of $A$, and denote it by
$\dim_BA$.

The definition of the box dimension of $A$ does not depend on the choice of the origin.

\begin{prop} Assume that $A$ is a given subset of $\eR^n$, and $0\notin\ov A$. Assume that also $w\notin\ov A$. Let $\Phi$ be the geometric inversion with respect to the origin, and $\Psi$ the geometric reflection with respect to the point $w\in\eR^n$, that is, $\Psi(x)=\frac{x-w}{|x-w|^2}$. Then  $f=\Psi\circ\Phi:\Phi(A)\to\Psi(A)$ is a bi-Lipschitz mapping. In particular, see \cite{falc}, we have $\ov\dim_B\Phi(A)=\ov\dim_B\Psi(A)$, and similarly for the lower box dimension.
\end{prop}

\begin{proof}
Let us first show that $f$ is Lipschitzian. It suffices to show that
\bgeq\label{f'}
\sup_{x\in\Phi(A)}\|f'(x)||<\ty.
\endeq 
As the matrix norm $\|\,\cdot\,\|$ we take any operator norm, say $\ty$-norm.
First, by direct computation we see that
\bgeq
\Phi'(x)=\frac{|x|^2I-2x\otimes x}{|x|^4},\nonumber
\endeq
where $I$ is the identity matrix, and $x\otimes x=x\cdot x^{\top}$, with $x$ understood as a column vector. 
Now, denoting $a=\Phi(x)$ we have:
\bgeqn
f'(x)&=&\Psi'(\Phi(x))\cdot\Phi'(x)\nonumber\\
&=&\frac{|a-w|^2I-2(a-w)\otimes (a-w)}{|a-w|^{4}}\cdot \frac{|a|^{-2}I-2\frac{a}{|a|^2}\otimes \frac{a}{|a|^2}}{|a|^{-4}}\nonumber\\
&=&\frac{|a|^2}{|a-w|^2}\left(I-\frac{a-w}{|a-w|}\otimes\frac{a-w}{|a-w|}\right)
\cdot\left(I-2\frac a{|a|}\otimes \frac a{|a|}\right).\nonumber
\endeqn
Therefore $\|f'(x)\|\le C\frac{|a|^2}{|a-w|^2}$, where $C$ is a positive constant and $a=\Phi(x)$. Since $a\notin\ov A$ and $\frac{|a|^2}{|a-w|^2}\to1$ as $a\to\ty$, the expression $\frac{|a|^2}{|a-w|^2}$ is bounded by a constant independent of $a$. This proves ~(\ref{f'}).

 On the other hand, $f^{-1}=(\Psi\circ\Phi)^{-1}=\Phi^{-1}\circ\Psi^{-1}=\Phi\circ\Psi$, and we can show in the similar way that $f^{-1}$ is Lipschitzian. Hence, $f$ is bi-Lipschitzian.
\end{proof}

We can define the upper and lower $s$-dimensional Minkowski contents of $A$ as the corresponding upper and lower Minkowski contents of $\Phi(A)$, $s\ge0$. We say that $A$ is Minkowski nondegenerate (Minkowski measurable) if $\Phi(A)$ is nondegenerate (Minkowski
measurable). 
\end{defn}

\begin{remark}
It is easy to get rid of  the condition for $A$ to be away from the origin. Indeed, if $A$ is any set in $\eR^n$, we can proceed as follows.
Define $A_1=A\cap B_1(0)$ and $A_2=A\setminus A_1$, and define
\bgeq
\ov\dim_BA=\max\{\ov\dim_BA_1,\ov\dim_B\Phi(A_2)\}.\nonumber
\endeq
It is easy to see that the upper box dimension so defined for unbounded sets satisfies the property of monotonicity (indeed, if $A\stq B$ then $\Phi(A)\stq\Phi(B)$, hence $\dim_BA=\dim_B\Phi(A)\le\dim_B\Phi(B)=\dim_BB$), and the property of finite stability. See \cite{falc}.
\end{remark}

Another basic property, as expected, is that the box dimension is preserved for unbounded sets with positive distance from the origin that are bi-Lipschitz equivalent. 

\begin{theorem}\label{bilip0}
Let $f:V_1\to V_2$ be a bi-Lipschitz map, where $0\notin \ov V_1$ and $0\notin\ov V_2$. If $A\st V_1$, then $\ov\dim_BA=\ov\dim_Bf(A)$, and analogously, $\underline\dim_BA=\underline\dim_Bf(A)$.
\end{theorem}

\begin{remark} It is possible to construct
a set $A$ in $\eR^n$ such that $\ov\dim_BA=n$ and $\underline\dim_BA=0$, see \cite{minkows}.
\end{remark}

Theorem~\ref{bilip0} follows immediately from the following proposition, the proof of which we postpone.

\begin{prop}\label{bilip}
Let $V_{1,2}$ be two neighbourhoods of $\ty$ in $\eR^n$ such that $0\notin \ov V_1$ and $0\notin\ov V_2$. The mapping $f:V_1\to V_2$ is bi-Lipschitzian if and only if the mapping 
$g:\Phi(V_1)\to \Phi(V_2)$ defined by 
\bgeq\label{ffi}
g(x)=(\Phi\circ f\circ\Phi)(x)
\endeq
is bi-Lipschitzian. 
\end{prop}

\begin{proof}[Proof of Theorem~\ref{bilip0}]
We have that $\Phi(A)\st\Phi(V_1)$, and using Proposition~\ref{bilip} we obtain that $\Phi(A)\simeq g(\Phi(A))$, with $g$ defined by (\ref{ffi}). Hence,
\bgeqn
\ov\dim_B A&=&\ov\dim_B\Phi(A)=\ov\dim_B g(\Phi(A))=\ov\dim_B(\Phi\circ f\circ\Phi^2)(A)\nonumber\\
&=&\ov\dim_B\Phi(f(A))=\ov\dim_B f(A).\nonumber
\endeqn
In the last equality we exploited the property of bi-Lipschitz invariance of the upper box dimension for bounded sets.
Analogously for the lower box dimension.
\end{proof}

To prove Proposition \ref{bilip}, we start with the following two lemmas.

\begin{lemma}\label{id}
For any $a,b\in\eR^{n}\setminus\{0\}$ we have
\bgeq\label{id0}
|\Phi(a)-\Phi(b)|=\frac{|a-b|}{|a|\,|b|},\nonumber
\endeq
where $|\,\,|$ is the Euclidean norm.
\end{lemma}

\begin{lemma}\label{bdd}
Let $f:V_1\to V_2$ be a bi-Lipschitz map, where $0\notin \ov V_1$ and $0\notin \ov V_2$. Then there exist two positive constants $C_1$ and $C_2$ such that
for all $a\in V_1$,
\bgeq\label{faa}
C_1\le \frac{|f(a)|}{|a|}\le C_2.\nonumber
\endeq
\end{lemma}
We omit the proofs of the previous lemmas and proceed to prove the proposition.
\begin{proof}[Proof of Proposition~\ref{bilip}]
Assume that $f$ is bi-Lipschitzian, i.e.\ $|f(a)-f(b)|\simeq|a-b|$ for all $a,b\in V_1$. Let $x=\Phi(a)$ and $y=\Phi(b)$ be any two elements from $\Phi(V_1)$.
Using Lemma~\ref{id} we have
$$
|f(\Phi(x))-f(\Phi(y))|\simeq|\Phi(x)-\Phi(y)|=\frac{|x-y|}{|x|\,|y|}=|x-y|\,|\Phi(x)|\,|\Phi(y)|.
$$
Therefore,
$$
\frac{|f(\Phi(x))-f(\Phi(y))|}{|f(\Phi(x))|\,|f(\Phi(y))|}\simeq|x-y|\frac{|\Phi(x)|\,|\Phi(y)|}{|f(\Phi(x))|\,|f(\Phi(y))|}.
$$
Applying Lemma~\ref{id} on the left-hand side, and Lemma~\ref{bdd} on the right-hand side, we obtain
$$
|\Phi(f(\Phi(x)))-\Phi(f(\Phi(y)))|\simeq|x-y|,
$$
i.e.\ $|g(x)-g(y)|\simeq|x-y|$ for all $x,y\in \Phi(V_1)$.

The proof of the converse implication is similar, and therefore we omit it.
\end{proof}

If $\C$ is a smooth curve (typically, an unbounded spiral) in $\eR^n$ converging to infinity, which does not pass through the origin, then we can define its
Minkowski content as follows. Assume that $d=\dim_B\C$ is well defined. Then we define  $\M^d(\C)=\M^d(\Phi(\C))$. Note that the right-hand side is well defined, since 
the set $\Phi(\C)$ is bounded. Furthermore, if we remove from $\C$ a portion of finite length, then the remaining part $\C_r$ has the same $d$-dimensional Minkowski content as $\C$. 
This is due to the excision lemma, see \cite[Lemma 5.6]{minkows}.

\begin{example}\label{a-string}
Let $\a$ be a given positive real number, and $A=\{k^{\a}:k\in\eN\}$. It is well known that the box dimension of $\Phi(A)=\{k^{-\a}:k\in\eN\}$ is equal to $1/(1+\a)$, see e.g.\
\cite{lapidus}. Therefore,
\bgeq
\dim_BA=\frac1{1+\a}.\nonumber
\endeq
\end{example}

The above example is a special case of the following result dealing with monotone strings $\cal L=(l_j)$ of infinite length, i.e.\ sequences of positive real numbers such that
$\sum_{j=1}^{\ty}l_j=\ty$ and $(l_j)$ is nonincreasing. We do not require that $l_j\to0$ as $j\to\ty$. This string is associated with an unbounded sequence of real numbers
$A=(a_k)$ defined by $a_k=\sum_{j=1}^kl_j$. Conversely, it is clear that a nondecreasing, unbounded sequence of real numbers $A=(a_k)$ defines the string $\cal L=(l_j)_j$, where 
$l_j=a_{j+1}-a_{j}$, and the string $\cal L=(l_j)$ is monotone if we require that $l_j$ is nonincreasing. Note that here the set $\Phi(A)=\{a_k^{-1}:k\in\eN\}$ is bounded, so that the classical box dimension makes sense. If we denote 
\bgeq\label{mu}
\mu_k=a_k^{-1}-a_{k+1}^{-1},\q\cal L'=(\mu_k)
\endeq 
then
\bgeq\label{B}
\ov\dim_B\Phi(A)=\ov\dim_B\cal L':=\inf\{\c>0:\sum_{j=1}^\ty\mu_j^{\c}<\ty\}.
\endeq
See \cite{lapidus}, where the right-hand side of (\ref{B}) is taken as the definition of the upper box dimension of a general bounded, monotone string $\cal L'=(\mu_k)$,
denoted by $D_{\cal L'}$ in this reference. Note that since $\sum_{j=1}^\ty\mu_j<\ty$, then $\ov\dim_B\Phi(A)\le1$. The following simple lemma provides a sufficient condition for a
string associated with the geometric inverse of an unbounded set to be monotone.

\begin{lemma}\label{string}
Let $A=(a_k)$ be an unbounded, monotonically nondecreasing sequence of positive numbers. The string $\cal L'=(\mu_k)$, defined by (\ref{mu}), 
is monotone if and only if for each $k\ge1$,
\bgeq\label{2}
\frac{a_{k+1}}{a_k}+\frac{a_{k+1}}{a_{k+2}}\ge2.
\endeq
Furthermore,
\bgeq
\ov\dim_BA=\ov\dim_B\cal L'.\nonumber
\endeq
\end{lemma}

\begin{proof} It is easy to check that $\mu_{k+1}\le\mu_{k}$ is equivalent with (\ref{2}). 
\end{proof}

\begin{example} For $a_k=k^{\a}$, where $\a$ is a fixed positive number, the condition (\ref{2}) is fulfilled,
since
$$
\frac{b_{k+1}}{b_k}+\frac{b_{k+1}}{b_{k+2}}=(1+k^{-1})^\a+(1+(k+1)^{-1})^{-\a}>(1+k^{-1})^\a+(1+k^{-1})^{-\a}>2,
$$
where the last inequality follows from the elementary inequality $t+t^{-1}>2$ for $t>1$.
Therefore the conclusion of Example \ref{a-string} is a special case of Lemma~\ref{string}, since for $\a$-strings $\cal L'=(k^{-\a}-(k+1)^{-\a})_{k\ge1}$ we have
$\ov\dim_B\cal L'=1/(1+\a)$, see \cite{lapidus}.
\end{example}

\begin{example}
Let $\C$ be a spiral defined in polar coordinates by $r=\f^{-\a}$, where $\f\ge\f_0>0$, and $\a$ is a given positive constant.
Then $\Phi(\C)$ is an unbounded spiral defined by $r=\f^{\a}$, where $\f\ge\f_0>0$. We have
$$
\dim_B\Phi(\C)=\max\{1,\frac2{1+\a}\},
$$
see \cite[p.\  121]{tricot}. Note that the nucleus of the spiral $\C$ is concentrated near the origin,
so that the nucleus of $\Phi(\C)$ is concentrated at infinity. For a strict definition of the nucleus see \cite{tricot}, intuitively, it is the part where the $\e$-neighbourhood of the spiral selfintersects. See Figure \ref{unbounded_spiral}.
\end{example}

\begin{example}
Let $\a$ and $\b$ be two given positive constants. Let $A$ be defined as the union of two spirals $\C_1$ and $\C_2$, defined  in polar coordinates as follows:
$\C_1\dots r=\f^{-\a}$ when $\f>1$ (bounded spiral tending to the origin), while $\C_2\dots r=\f^{\b}$ when $\f>1$ (unbounded spiral, away from the origin).
It is easy to see, using finite stability of the box dimension, that
\bgeq
\dim_BA=\max\left\{1,\frac2{1+\min\{\a,\b\}}\right\}.\nonumber
\endeq
\end{example}

Starting from $\dot x=P(x)$, see (\ref{x}), using geometric inversion we arrived at
$\dot u=P^*(u)$ where
\bgeq\label{P*}
P^*(u)=|u|^2\tilde P(u)-2u(u\cdot\tilde P(u)).
\endeq
It is clear that $P^{**}=P$ for each vector field $P$, since the geometric inversion with respect to the origin is involutive.
It is easy to see that 
$$
{}^*:C^1(\eR^n\setminus\{0\},\eR^n)\to C^1(\eR^n\setminus\{0\},\eR^n)
$$ 
is a linear operator with real coefficients: 
for any $\g,\mu\in\eR$ and $F,G\in C^1(\eR^n\setminus\{0\},\eR^n)$ we have $(\g F+\mu G)^*=\g F^*+\mu G^*$.

\begin{remark}
If in (\ref{x}) $P(x)$ is a rational function (that is, the component functions are rational functions of $x_j$, $j=1,\dots,n$), 
then from (\ref{P*}) we see that $P^*(u)$ is also a rational function. 
The phase portrait of the system (\ref{x}) is the same (outside the origin) as for the polynomial system corresponding to $d(x)P(x)$, where $d(x)$ is the common denominator of all $P_j(x)$. Analogously for the system (\ref{u}).
\end{remark}

The following lemma deals with a special class of right-hand sides $P(x)$ of (\ref{x}) for which $P^*(u)$ can be easily computed.

\begin{lemma}\label{Rx}
Let us consider the system (\ref{x}) with $P(x)=Rx-\c x g(|x|)$, $x\in\eR^n$, where $\c$ is a real constant and $g:(0,\ty)\to\eR$ a continuous function,
and $R$ is an $n\times n$ real antisymmetric matrix: $R^\top=-R$.
Then 
\bgeq
P^*(u)=Ru+\c u g(|u|^{-1}).\nonumber
\endeq
\end{lemma}

\begin{proof} The matrix $R$ is antisymmetric if and only if $Rx\cdot x=0$ for all $x$.
The claim follows from (\ref{P*}) and $\tilde Pu=|u|^{-2}(Ru-\c u g(|u|^{-1}))$ after a short computation.
\end{proof}

\begin{example}
In particular, if $P(x)=Rx$, where $R$ is a real antisymmetric matrix, then $P^*(u)=Ru$, that is, $P=P^*$. If $P(x)=c x$, where $c$ is a real constant, then $P^*(u)=-cu$.
\end{example}

A typical example of a real matrix $R$ satisfying the condition $Rx\cdot x=0$ for all $x\in\eR^n$ is any diagonal block matrix containing either matrices of the form
\bgeq
\g_j
\begin{bmatrix}
0&-1\\
1&\phantom{-}0
\end{bmatrix}\nonumber
\endeq
on the diagonal (here $\g_j\in\eR$), or zeros. 

If we deal with an ODE in the complex phase space $\Ce^n$:
\bgeq
\dot z=P(z),\nonumber
\endeq
where $z=(z_1,\dots,z_n)^{\top}$, then introducing the new variable $u=z/|z^2|$, where $|z|^2=\sum_{j=1}^n|z_j|$, we obtain
\bgeq
\dot u=P^*(u),\nonumber
\endeq
where 
\bgeq
P^*(u)=|u|^2\tilde P(u)-2u\,\re(u\,|\,\tilde P(u)).\nonumber
\endeq
Here we define $(u\,|\,v)=\sum_{j=1}^n u_j\ov v_j$.

\begin{lemma}
Let $P(z)=Rz-\c z g(|z|)$, where $R=\diag(i\g_1,\dots,i\g_n)$, $\g_j\in\eR$, $\c$ is a given complex number and $g:(0,\ty)\to\ty$ is a continuous function. Then
\bgeq
P^{*}(u)=Ru+\ov\c ug(|u|^{-1}).\nonumber
\endeq
\end{lemma}

\begin{proof}
Note that $(Ru\,|\,u)=i\sum_j\g_j|u_j|^2$, so that $\re (Ru\,|\,u)=0$. The rest of the proof is the same as in the proof of Lemma~\ref{Rx}.
\end{proof}

We shall also often need the following technical lemma, dealing with planar systems of weak focus type near the origin:
\bgeq\label{pq}
\begin{aligned}
\dot x&=-y+p(x,y)\\
\dot y&=\phantom{-}x+q(x,y).
\end{aligned}
\endeq
A typical situation is when $p$ and $q$ are analytic functions with McLaurin series containing quadratic or higher order terms only.
 It is an extension of Lemma~\ref{Rx} in the case of $n=2$.

\begin{lemma}\label{wf}
The system obtained from (\ref{pq}) by geometric inversion is equal to
\bgeq\label{pquv}
\begin{aligned}
\dot u&=-v+(v^2-u^2)\tilde p-2uv\,\tilde q\\
\dot v&=\phantom{-}u+(u^2-v^2)\tilde q-2uv\,\tilde p,
\end{aligned}
\endeq
where $\tilde p=p(\frac u{u^2+v^2},\frac v{u^2+v^2})$ and $\tilde q=q(\frac u{u^2+v^2},\frac v{u^2+v^2})$.
\end{lemma}

The corresponding general result concerning the box dimension of spiral trajectories of (\ref{pquv})
can be seen in Theorem~\ref{analytic}.

\subsection{Definition of box dimension of unbounded sets in the plane using the Riemann sphere}
Assume that $A\st\eR^2$ is an unbounded set such that the origin is not its accumulation point, so that $\Phi(A)$ is bounded. 
Let $S$ be the Riemann sphere in $\eR^3$ of radius $1/2$ with the center at $(0,0,1/2)$.
Let $\pi_S:\eR^2\to S$ be the stereographic projection of the plane to the sphere. The following result shows that the box dimension of $A$ (defined via geometric inversion)
is the same as the the box dimension of its stereographic projection. Here the set $\pi_S(A)$ contained in the Riemann sphere, is viewed as a subset of $\eR^3$, i.e.\ its box dimension is
computed via its $\e$-neighbourhood in $\eR^3$.

\begin{prop}\label{riemann}
Let $A$ be an unbounded set in $\eR^2$ such that $0\notin\ov A$. Then
$\ov\dim_BA=\ov\dim_B\pi_S(A)$, analogously for the lower box dimension. Furthermore,
if $A$ is Minkowski nondegenerate, so is $\pi_S(A)$.
\end{prop}

\begin{remark}
If $A'$ is a subset of the Riemann sphere $S$, then it is possible to define its $\e$-neighbourhood on the manifold $S$ 
with respect to its metric and the resulting surface measure. This permits us to define the new (upper or lower) box dimension
in the usual way. It can be shown that the box dimension of $A'$ with respect to the manifold $S$ 
is the same as the box dimension of $A'$ with respect to $\eR^3$, i.e.\ with respect to its $\e$-neighbourhood in $\eR^3$.
\end{remark}

\begin{figure}[h]
\begin{center}
\includegraphics[trim=0cm 6cm 0cm 10cm,clip=true,width=12cm]{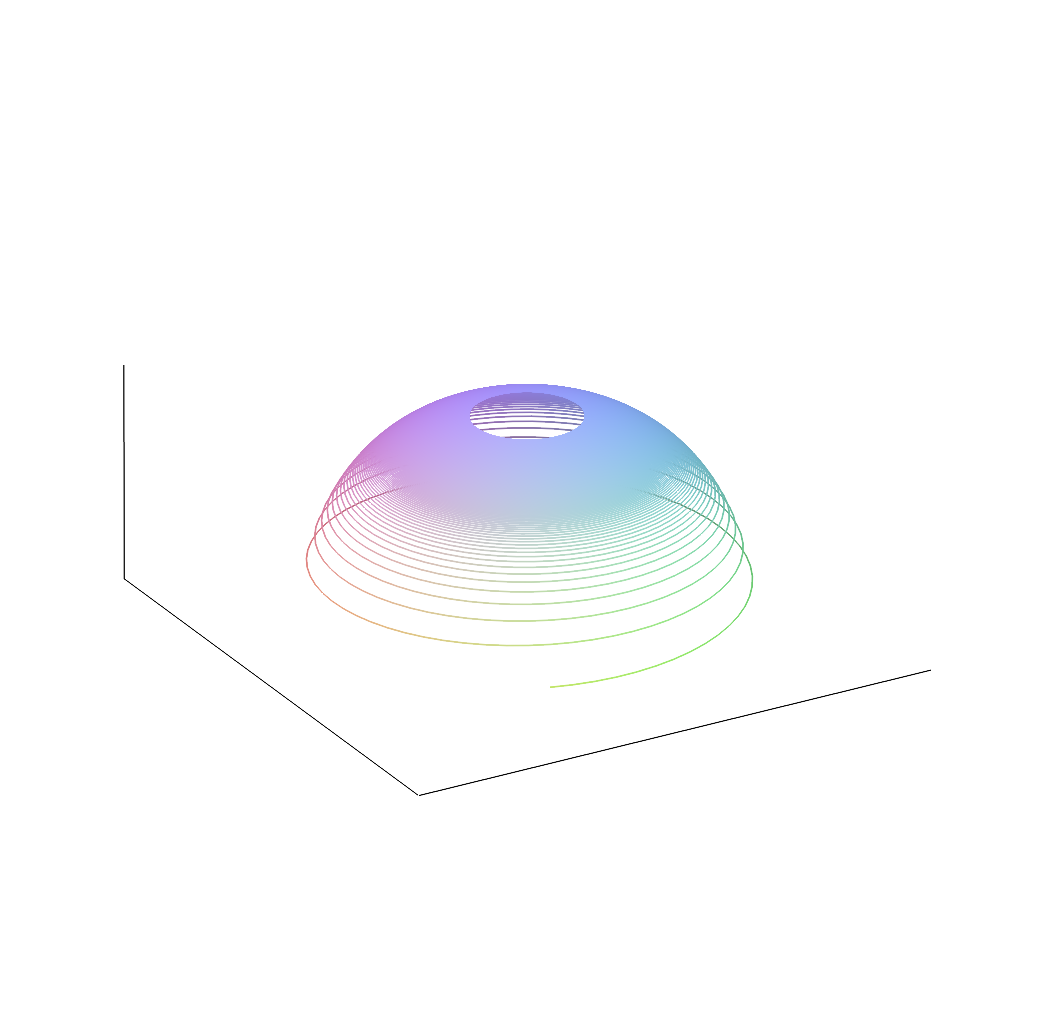}
\end{center}
\caption{The unbounded spiral $f(\varphi)=\varphi^{1/4}$ projected to the Riemann sphere of radius $1/2$.}
\label{rspiral}
\end{figure}

The proof of Proposition~\ref{riemann} rests on the following lemma.

\begin{lemma}\label{riemann_lip}
Let $B_{r_0}(0)$ be the disk of radius $r_0\in(0,1)$ in the plane, and $F:B_{r_0}(0)\to S$ defined as composition of inversion and stereographic projection, i.e.\ 
$F(x)=\pi_S(\frac{x}{|x|^2})$ for $x\ne0$, and $F(0)=N$ (the north pole of $S$). Then the mapping $F$ is bi-Lipschitzian.
\end{lemma}

\begin{proof}
Let $P:S\to\eR^2$ be the orthogonal projection. It is easy to see that $V=F(B_{r_0}(0))$ is a neighbourhood of the north pole $N=(0,0,1)$ on the Riemann sphere, and strictly above the equator, due to $r_0<1$.
It is then clear that the restriction of the projection $P|_V$ is a bi-Lipschitz mapping. Therefore, it suffices to show that $G=P\circ F$ is bi-Lipschitzian, since then $F=P^{-1}\circ G$ will be bi-Lipschitzian as a composition of bi-Lipschitz functions.

We pass to polar coordinates $(r,\f)$ in the plane. It is clear that $G$ has the form $G(r,\f)=(g(r),\f)$. Fixing any vertical semi-plane defined by given $\f$, 
it intersects the Riemann sphere in a semi-circle above the $r$-axis. For a given $r>0$, the strait line joining $r^{-1}$ with $N$ interesects the circle in the point
the horizontal component of which is equal to $g(r)$. An easy calculation shows that
\bgeq\label{g}
g(r)=\frac{r}{r^2+1}.
\endeq
Since $g'(r)=\frac{1-r^2}{(r^2+1)^2}$, we have that $g'(0)=1$ and $g'(r_0)>0$, hence $g'(r)\in(g'(r_0),1)$. This shows that $G:V\to G(V)$ is bi-Lipschitzian.
\end{proof}

\begin{remark}
It is easy to see that the bound $r_0<1$ in Lemma~\ref{riemann_lip} is optimal, since $g'(1)=0$.
\end{remark}

\begin{example} Let $r=\f^{-\a}$ be a given spiral $A$ in the plane, $\f\ge\f_0>0$. Then $G(A)$, with $G$ from the proof of Lemma~\ref{riemann_lip}, is a spiral defined by $\rho=\frac{\f^{-\a}}{\f^{-2\a}+1}$, and
both of them are bi-Lipschitz equivalent. In particular, they both have the box dimension equal to $d=\max\{1,\frac2{1+\a}\}$. The following proposition shows that they both have
the same $d$-dimensional Minkowski content (its value has been computed in \cite{zuzu}, see (\ref{mmink}) below).
\end{example}

\begin{prop}\label{mink_same} Assume that $\a\in(0,1)$ is a given constant.
Let $r=f(\f)$ be a spiral such that $f(\f)\sim C\f^{-\a}$, $f'(\f)\sim -\a C\f^{-\a-1}$ as $\f\to\ty$, and there exists $M>0$ such that $|f''(\f)|\le M\f^{-\a}$,
for all $\f\ge\f_0>0$. Define the new spiral $\rho=\frac{f(\f)}{f(\f)^2+1}$. Then both spirals have the same box dimension $d=2/(1+\a)$, and the same $d$-dimensional 
Minkowski contents.
\end{prop}

The proof of this proposition is a direct consequence of \newline \cite[Theorem 6]{zuzu} that we state here in a simplified, but equivalent form (stated also in \cite[Theorem 3]{clothoid}).

\begin{theorem}\label{mink}
{\rm (Minkowski measurable spirals)} Assume that $f:[\f_1,\ty)\to(0,\ty)$ is a decreasing $C^2$-function, and $\f_1>0$. Assume that there exists the limit
\bgeq
m=\lim_{\f\to\ty}\frac{f'(\f)}{(\f^{-\a})'},\nonumber
\endeq
and $m>0$. Let there be a positive constant $M$ such that $|f''(\f)|\le M\f^{-\a}$ for all $\f\ge\f_1$. Let $\C$ be the graph of the spiral $\rho=f(\f)$ with $\a\in(0,1)$ and $d=2/(1+\a)$.
Then $\dim_B\C=d$, the spiral is Minkowski measurable, and moreover,
\bgeq\label{mmink}
\M^d(\C)=m^d\pi(\pi\a)^{-2\a/(1+\a)}\frac{1+\a}{1-\a}.
\endeq
\end{theorem}

Now the proof of proposition \ref{mink_same} follows by showing that the functon $h(\f)=\frac{f(\f)}{f(\f)^2+1}$ has the same properties as $f(\f)$ in Theorem~\ref{mink},
eventually with a different value of $M>0$.

\begin{remark}
The analogous construction as in Proposition~\ref{riemann} and in Lemma~\ref{riemann_lip} can be performed starting with $\eR^n$ instead of $\eR^2$, and using the Riemann sphere $S^n$ in $\eR^{n+1}$ of radius $1/2$, centered at $(0,\dots,0,1/2)$.
It suffices to use the stereographic projection $\pi_{S^n}:\eR^n\to S^n$.
\end{remark}

\begin{example}
Let us suppose that the Riemann sphere is of radius $R$. Proposition \ref{riemann} still holds which is easy to see using analogous arguments as in Lemma \ref{riemann_lip}. In this case we take $r_0\in(0,1/(2R))$ and get $$g(r)=\frac{4R^2r}{4R^2r^2+1}.$$ On the other hand, the Minkowski content of the new spiral $\rho$ will be afected with the radius of the Riemann sphere. Concretely: $$\mathcal{M}^d(\Gamma_\rho)=(4R^2)^d\mathcal{M}^d(\Gamma_r).$$
\end{example}

Let $S^2$ be the Poincar\' e sphere in $\eR^3$ of radius $R$, i.e. $S^2=\{(X,Y,Z)\in\eR^3\ \colon \ X^2+Y^2+Z^2=R^2\}$. We shall project onto the sphere from the $(x,y)$ plane 
placed tangentially at the north pole. We are interested how the box dimension of a focus type spiral is affected with geometric inversion and projection onto the Poincar\'e sphere.

\begin{prop}\label{p_sfera}
Let $\Gamma_1\ldots r=f(\varphi)$ be a spiral of focus type such that 
$$
f(\varphi)\simeq\varphi^{-\a},\ \ |f'(\varphi)|\simeq\varphi^{-\a-1},\ \ |f''(\varphi)|\leq M\varphi^{-\a}
$$
as $\f\to\ty$, for some positive constants $\a$ and $M$. Firstly we geometrically invert this spiral, and then project it on $S^2$. We shall denote this new spiral in $\eR^3$ with $\Gamma_2$.
Then
\bgeq\label{n_dim}
\dim_B\Gamma_2=\frac{2+\a}{1+\a}.
\endeq
\end{prop}

\begin{proof}[Proof.]
In cylindrical coordinates in $\eR^3$ the above described map of the spiral $\Gamma_1$ is given by
$$
(f(\varphi),\varphi)\mapsto\left(\frac{R}{\sqrt{1+R^2f(\varphi)^2}},\varphi,\frac{R^2f(\varphi)}{\sqrt{1+R^2f(\varphi)^2}}\right).
$$
We will use \cite[Theorem 5(b)]{3}. In our case we have $r=R-F(\f)$,
where 
$$
F(\f)=R-\displaystyle\frac{R}{\displaystyle\sqrt{1+R^2f(\varphi)^2}}, 
$$
and $Z=\sqrt{R^2-r^2}=g(|R-r|)$, for $g(t)=\sqrt{2Rt-t^2}$.

Now we have to check the conditions from \cite[Theorem 5(b)]{3}:
$$
g(t)\displaystyle\sim t^{1/2}\sqrt{2R},\ 
g'(t)\sim\sqrt{\frac{R}{2}}t^{-1/2},\ 
g''(t)\sim -\frac{\sqrt{R}}{2\sqrt{2}}t^{-3/2}.
$$
We can see that $g$ meets the conditions for $\b=1/2$. Now we check the conditions on $F$:
$$
F(\varphi)\sim\frac{R^3}{2}f(\varphi)^2\simeq\varphi^{-2\a};
$$
$$
|F'(\varphi)|=\frac{R^3f(\f)|f'(\f)|}{(1+R^2f(\f)^2)^{3/2}}\simeq\varphi^{-2\a-1};
$$
$$
|F''(\varphi)|\le \frac{R^3}{(1+R^2f^2)^{5/2}}\left(|f'|^2+f|f''|+R^2f^3|f''|+2R^2f^2|f'|^2\right)\leq c\f^{-2\a}
$$ for $\f$ sufficiently large.
So we can see that $F$ meets the conditions of the \cite[Theorem 5(b)]{3} for $\a_2=2\a$ and we have the conclusion
$$
\dim_B\Gamma_2=\frac{2+\a_2\b}{1+\a_2\b}=\frac{2+\a}{1+\a}.
$$
\end{proof}
\begin{remark}
For $\a\geq 1$ the box dimension of the spiral $\Gamma_1$ is one. This shows that projecting a spiral of trivial box dimension to the Poincar\' e sphere we can get a spiral of nontrivial box dimension.
\end{remark}

\begin{figure}[h]
\begin{center}
\includegraphics[trim=0cm 6cm 0cm 10cm,clip=true,width=12cm]{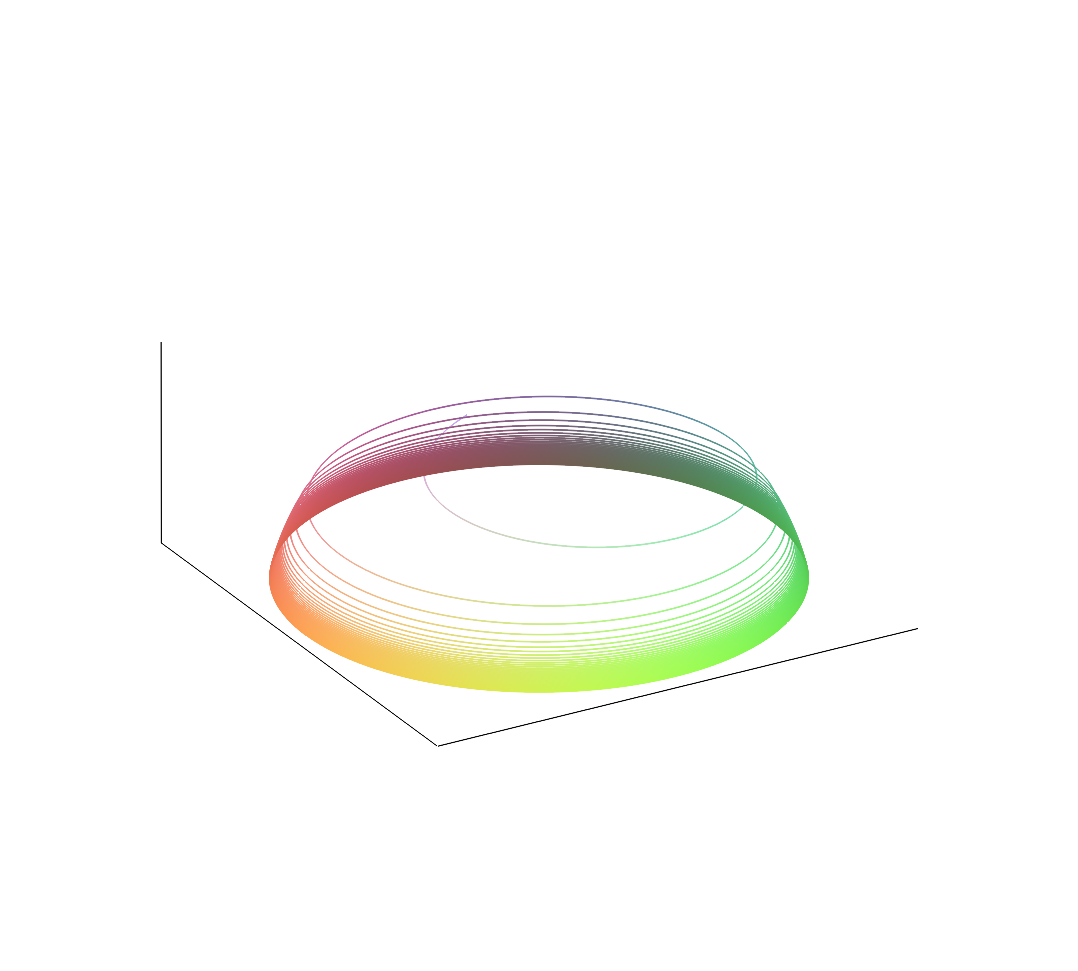}
\end{center}
\caption{The unbounded spiral $f(\varphi)=\varphi^{1/4}$ projected to the Poincar\' e sphere of radius $1$.}
\label{pspiral}
\end{figure}

\begin{remark}\label{dim_trans}
As the Poincar\' e sphere is usually represented by orthogonally projecting it on the $xy$-plane, this will further affect the box dimension of a spiral defined in Proposition \ref{p_sfera}. After projecting it, we will have a limit cycle type spiral $\Gamma_3\ldots r=R-F(\f)$ in the plane and its box dimension will be reduced to $$\dim_B\Gamma_3=\frac{2+2\a}{1+2\a}.$$

For a focus type spiral $\Gamma_1$ of nontrivial box dimension i.e. for $\a\in(0,1)$ we have the next relations between box dimensions:
\bgeq
\begin{aligned}
\dim_B\Gamma_2&=1+\frac{1}{2}\dim_B\Gamma_1\\
\dim_B\Gamma_3&=\frac{2}{3-\dim_B\Gamma_2},\q
\dim_B\Gamma_3&=\frac{4}{4-\dim_B\Gamma_1}.\nonumber
\end{aligned}
\endeq

\end{remark}

\section{Weak focus at infinity}

\subsection{Weakly damped oscillator}
Let us consider a weakly damped oscillator 
\bgeq
\ddot y+Cy^{\a}(\dot y)^{\b}+y=0,\nonumber
\endeq
where $\a$ is even and $\b$ odd positive integer, and $C$ is any positive constant.
 It is well known that it is globally stable. It is equivalent to the following planar system:
\bgeq
\begin{aligned}\label{damped}
\dot x&=-y-Cx^\b y^\a  \\
\dot y&=\phantom{-}x.\nonumber
\end{aligned}
\endeq

All nontrivial trajectories $\C$, corresponding to $t\ge0$, are spirals converging clockwise to the origin, and the origin is the weak focus.
Using Lemma~\ref{wf} we conclude that the corresponding system obtained by geometric inversion is
\bgeq\label{damped2}
\begin{aligned}
\dot u&=-v+C\frac{u^\b u^\a(u^2-v^2)}{(u^2+v^2)^{\a+\b}}  \\
\dot v&=\phantom{-}u+2C\frac{u^{\b+1}v^{\a+1}}{(u^2+v^2)^{\a+\b}}.
\end{aligned}
\endeq
All trajectories corresponding to $t\ge0$, starting outside the origin, are of the form $\Phi(\C)$ for some $\C$ as above. 
The spirals $\Phi(\C)$ are converging clockwise to infinity, and the infinity is the weak focus.
The system (\ref{damped2}) becomes polynomial after multiplying the right-hand sides by $(u^2+v^2)^{\a+\b}$:
\bgeq\label{damped2pol}
\begin{aligned}
\dot u&=-v(u^2+v^2)^{\a+\b}+Cu^\b u^\a(u^2-v^2)  \\
\dot v&=\phantom{-}u(u^2+v^2)^{\a+\b}+2Cu^{\b+1}v^{\a+1}.
\end{aligned}
\endeq

\begin{theorem}
Any nontrivial trajectory $\C$ of the system (\ref{damped2pol}), corresponding to $t\ge0$, is a spiral converging to infinity, and
\bgeq
\dim_B\C=2\left(1-\frac1{\a+\b}\right).\nonumber
\endeq
Furthermore, the spirals are Minkowski nondegenerate.
\end{theorem}

This follows immediately from \cite[Theorem 7]{pzz}. 

\subsection{Li\'enard systems}

Let us consider the following Li\'enard system:
\bgeq\label{lien}
\begin{aligned}
\dot x&=\displaystyle-y+\sum_{i=1}^{N}a_{2i}x^{2i}+\sum_{i=k}^{N}a_{2i+1}x^{2i+1}\\
\dot y&=\phantom{-}x.\nonumber
\end{aligned}
\endeq
Here we assume that $a_{2k+1}\ne0$, which means that this is the first nontrivial coefficient on the right-hand side having odd index.
The system obtained by geometric inversion is
\bgeq\label{lien1}
\begin{aligned}
\dot u&=\displaystyle-v+(v^2-u^2)\tilde p(u)\\
\dot v&=\phantom{-}u-2uv\,\tilde p(u),
\end{aligned}
\endeq
where 
\bgeq
\tilde p(u)=\sum_{i=1}^{N}a_{2i}\frac{u^{2i}}{(u^2+v^2)^{2i}}+\sum_{i=k}^{N}a_{2i+1}\frac{u^{2i+1}}{(u^2+v^2)^{2i+1}}.\nonumber
\endeq
Multiplying the right-hand sides of (\ref{lien1}) by $(u^2+v^2)^N$, the system becomes polynomial, retaining the same phase portrait outside the origin.
An immediate consequence of \cite[Theorem~6]{belg} is the following result.

\begin{theorem}
If in (\ref{lien1}) we have $a_{2k+1}\ne0$, then for any initial point $(u_0,v_0)\ne(0,0)$ we have
that the corresponding trajectory $\C=\C(u_0,v_0)$ starting from that point is a spiral tending to infinity, and
$$
\dim_B\C=2\left(1-\frac{1}{2k+1}\right).
$$
Furthermore, $\C$ is Minkowski nondegenerate.
\end{theorem}

As we see, unbounded spiral trajectories (in fact, semitrajectories, i.e.\ starting from initial point) of Li\'enard systems can achieve box dimensions with values from the following set only:
\bgeq
D_0=\{\frac{4k}{2k+1}:k\in\eN\}=\{\frac43,\frac85,\frac{12}7,\frac{16}9,\frac{20}{11},\dots\}.\nonumber
\endeq
For analytic systems, these are the only values of box dimensions that unbounded spiral trajctories can achieve, see \cite{belg}.

A more general result can be stated in terms of the Poincar\'e map at infinity. We deal with the system (\ref{pq}) such that $p(x,y)$ and $q(x,y)$
are analytic functions of the form
\bgeq\label{pqan}
p(x,y)=\sum_{k=2}^\ty p_k(x,y),\q q(x,y)=\sum_{k=2}^\ty q_k(x,y)
\endeq
where $p_k$ and $q_k$ are homogeneous polynomials of $k$-th degree. The Lyapunov coefficient of a system at infinity is defined as the Lyapunov coefficient at the origin
of the system obtained by geometric inversion. The Lyapunov coefficient near the weak focus is defined as the coefficient of the leading term of the Taylor expansion of the displacement function. The following result follows immediately from \cite[Theorem~6]{belg}.

\begin{theorem}\label{analytic}
Let $\C$ be an unbounded spiral trajectory, away from the origin, associated to the system (\ref{pquv}). Assume that $p(x,y)$ and $q(x,y)$
are analytic functions as in (\ref{pqan}). If $V_{2k+1}$ is the first nonzero Lyapunov coefficient of (\ref{pquv}) at infinity, then
\bgeq
\dim_B\C=2\left(1-\frac1{2k+1}\right).\nonumber
\endeq
Furthermore, $\C$ is Minkowski nondegenerate.
\end{theorem}

\subsection{Classical Hopf bifurcation}
The classical Hopf bifurcation is defined by the following system for $k=1$:
\bgeq\label{hopf}
\begin{aligned}
\dot x&=-y-x\left((x^2+y^2)^k+a\right)  \\
\dot y&=\phantom{-}x-y\left((x^2+y^2)^k+a\right),
\end{aligned}
\endeq
where $a$ is the bifurcation parameter. The corresponding spirals $\C$ are converging clockwise to the origin. 
Using Lemma~\ref{Rx} (here $R$ is the symplectic $2\times2$ matrix and $g(r)=r^{2k}+a$) we have that the related system obtained
from (\ref{hopf}) by geometric inversion is
\bgeq\label{hopfuv}
\begin{aligned}
\dot u&=-v+u\left((u^2+v^2)^{-k}+a\right)  \\
\dot v&=\phantom{-}u+v\left((u^2+v^2)^{-k}+a\right),
\end{aligned}
\endeq
and the corresponding spirals are converging clockwise to infinity, which is the weak focus.
The corresponding polynomial system 
\bgeq\label{hopfuvp}
\begin{aligned}
\dot u&=-v(u^2+v^2)^{k}+u\left(1+a(u^2+v^2)^{k}\right)  \\
\dot y&=\phantom{-}u(u^2+v^2)^{k}+v\left(1+a(u^2+v^2)^{k}\right),
\end{aligned}
\endeq
has the same phase portrait as (\ref{hopfuv}) outside the origin.

In polar coordinates $(r,\f)$ system (\ref{hopf}) has the form
\bgeq\label{hopfpolar}
\begin{aligned}
\dot r&=-r(r^{2k}+a)  \\
\dot \f&=1.
\end{aligned}
\endeq
For $a<0$ the limit cycle is born off the origin, $r=(-a)^{1/k}$,
while system (\ref{hopfuv}) in polar coordinates $(\rho,\f)$ has the form
\bgeq\label{hopfuvpolar}
\begin{aligned}
\dot \rho&=\rho(\rho^{-2k}+a) \\
\dot \f&=1.
\end{aligned}
\endeq
In this case, for $a<0$ the limit cycle is born off infinity. Here $r=(-a)^{-1/k}$ and
$r\to\ty$ as $a\to0_-$. The system (\ref{hopfuvpolar}) is clearly the same  as the one obtained from (\ref{hopf}) by introducing the coordinates
$(\rho,\f)$ defined via
$x=\frac{\cos\f}\rho$, $y=\frac{\sin\f}\rho$. The following result shows that the box dimension `recognizes' the Hopf bifurcation.

\begin{theorem}
Let $a=0$ in the bifurcation problem (\ref{hopfuv}) or (\ref{hopfuvp}). Then any unbounded spiral trajectory $\C$, away of the origin, has the box dimension
equal to
\bgeq
\dim_B\C=\frac{4k}{2k+1},\nonumber
\endeq
and is Minkowski measurable. For all the other values of $a$ the box dimension is trivial, i.e.\ equal to $1$.
\end{theorem}

This is an immediate consequence of \cite[Theorem~7]{zuzu}.

\section{Hopf-Takens bifurcation at infinity}

Using geometric inversion and results from \cite{zuzu}, chapter 4, we shall study fractal properties of Hopf-Takens bifurcation occurring at infinity. For a standard generic Hopf-Takens bifurcation we have the  normal form: 
\bgeq
\begin{aligned}
X_\pm^{(l)}&:=\left(-y\frac{\partial}{\partial x}+x\frac{\partial}{\partial y}\right)\\
&\phantom{:=}\,\,\pm\left((x^2+y^2)^l+a_{l-1}(x^2+y^2)^{l-1}+\cdots+a_0\right)ÿ\left( x\frac{\partial}{\partial x}+y\frac{\partial}{\partial y}\right),\nonumber
\end{aligned}
\endeq
where $(a_0,\ldots,a_{l-1})\in\eR^l$ is fixed.
In sequel we will consider $X_+^{(l)}$ only, since the case $X_-^{(l)}$ is treated similarly. In case $X_+^{(l)}$ the normal form in polar coordinates is given by  
\bgeq\label{htpolar}
\begin{aligned}
\dot r&=r\left(r^{2l}+\sum_{j=0}^{l-1}a_jr^{2j}\right)\\
\dot \f&=1.
\end{aligned}
\endeq
Geometric inversion with $\rho:=1/r$ yields a new system of differential equations
\bgeq\label{htinf}
\begin{aligned}
\dot \rho&=-\rho\left(\rho^{-2l}+\sum_{j=0}^{l-1}a_j\rho^{-2j}\right)\\
\dot \f&=1.
\end{aligned}
\endeq

Now it is easy to see that the following analogous versions of \newline \cite[Theorems 9 and 10]{zuzu} are valid.

\begin{theorem}\label{focus_ty} {\rm(The case of focus).}
Let $\Gamma$ be a part of a trajectory of (\ref{htinf}) near infinity.

(a) Assume that $a_0\neq 0$. Then the spiral $\Gamma$ is of exponential type, that is, comparable with $\rho=e^{-a_0\f}$, and hence $\dim_B\Gamma=1$.

(b) Let $k$ be fixed, $1\leq k\leq l$, $a_0=\cdots=a_{k-1}=0$, $a_k\neq 0$. Then $\Gamma$ is comparable with the spiral $\rho=\f^{1/2k}$, and
\bgeq
\dim_B\Gamma=\frac{4k}{2k+1}.\nonumber
\endeq
\end{theorem}

\begin{theorem}\label{lch}{\rm (The case of limit cycle).} Let the system (\ref{htinf}) have a limit cycle $\rho=a$ of multiplicity $m$, $1\leq m\leq l$. By $\Gamma_1$ and $\Gamma_2$ we denote the parts of two trajectories of (\ref{htinf}) near the limit cycle from outside and inside respectively. Then the trajectories $\Gamma_1$ and $\Gamma_2$ are comparable

(a) with exponential spirals $\rho=a\pm e^{-\b\f}$ of limit cycle type when $m=1$, for some constants $\b\neq 0$ (depending only on the coefficients $a_i$, $0\leq i\leq l-1$),

(b) with power spirals $\rho=a\pm\f^{-1/(m-1)}$ when $m>1$.

In both cases we have
\bgeq
\dim_B\Gamma_i=2-\frac{1}{m},\ \ i=1,2.\nonumber
\endeq
\end{theorem}

\begin{remark}
In Theorem \ref{lch} the parts of trajectories we are observing are contained in an open ring around the limit cycle which is a bounded set that does not contain the origin. As geometric inversion $\Phi$ is bi-Lipschitzian on such sets, Theorem \ref{lch} is a direct consequence of Theorem 10 from \cite{zuzu}. 
\end{remark}

\begin{remark}
From Theorem~\ref{lch} we know that for 
 (\ref{htinf}) each spiral trajectory of limit cycle type has box dimension from the set
\bgeq
D_1=\{2-\frac1m:m\in \eN\}=\left\{1,\frac32,\frac53,\frac74,\frac95,\dots\right\}.\nonumber
\endeq
See \cite[p.\ 958]{belg}.
\end{remark}

Let us have a look at the inversion of a standard Hopf bifurcation in polar coordinates, i.e. system (\ref{htinf}) for $l=1$:
\bgeq\label{htinfl1}
\begin{aligned}
\dot \rho&=-\rho(\rho^{-2}+a_0)\\
\dot \f&=1.
\end{aligned}
\endeq 

\begin{figure}[b]
\begin{center}
\includegraphics[height=3.9cm]{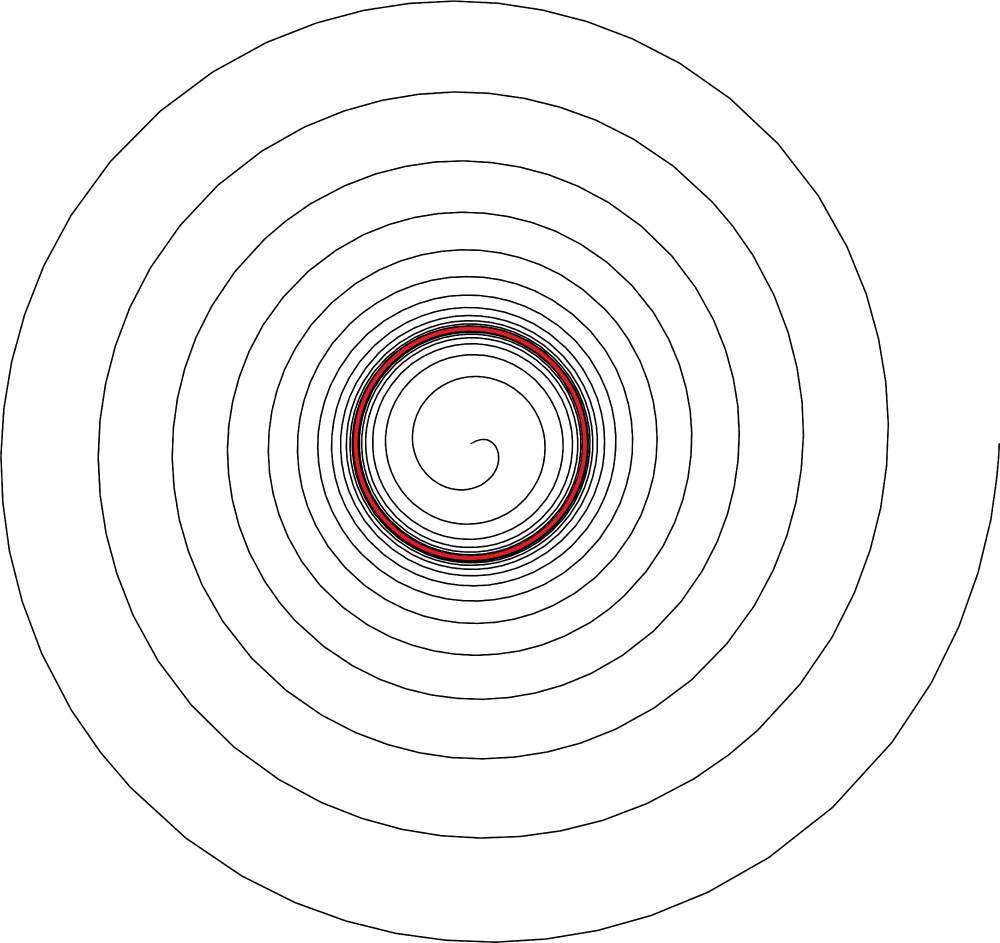}
\includegraphics[height=3.9cm]{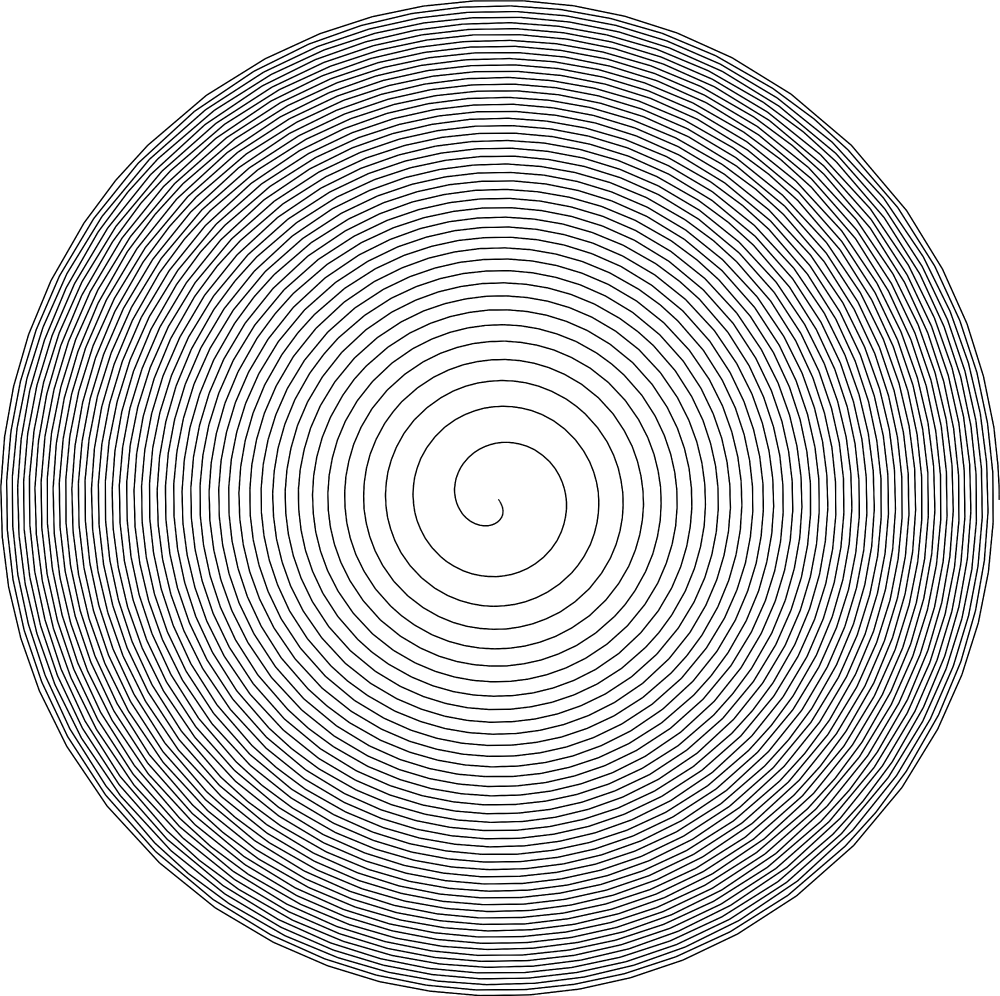}
\includegraphics[height=3.9cm]{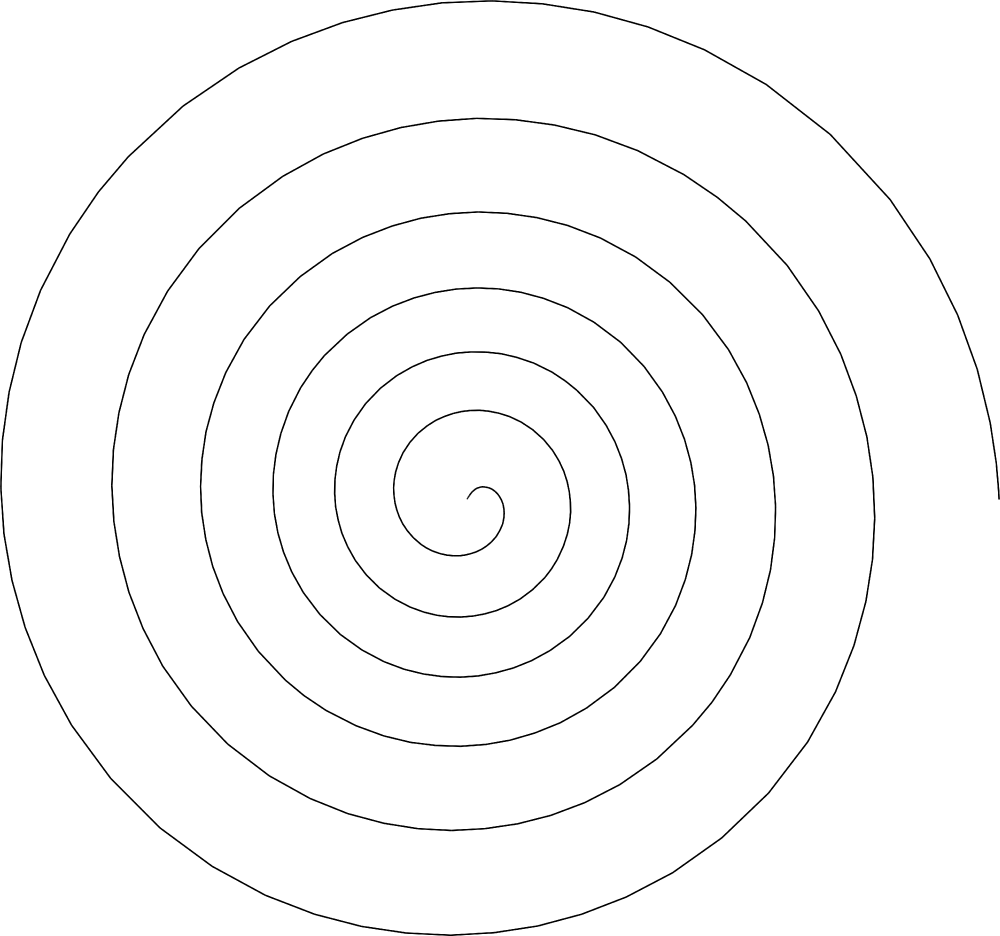}
\end{center}
\caption{Trajectories of the system (\ref{htinfl1}). Left: $a_0=-1/25$ with a limit cycle born from infinity; Middle: $a_0=0$ with weak focus at infinity; Right: $a_0=1/25$ with exponential spiral at infinity.}
\label{inf_hopf_fig}
\end{figure}

Viewing $a_0$ as a bifurcation parameter, we have the three following possibilities.

(1) For $a_0<0$ the trajectories of (\ref{htinfl1}) are given with 
$$
\rho(\f)=-\frac{1}{a_0}\sqrt{-a_0+a_0^2Ce^{-2a_0\f}},\ \ C\geq 0,\ \f\in\eR\textrm{ or } C<0,\ \f\leq\frac{\ln(a_0C)}{2a_0}.
$$
We can see that we have a strong focus at infinity and the circle $\rho=(-a_0)^{-1/2}$ is the limit cycle for trajectories from inside and outside near the circle. The corresponding spirals near infinity are comparable with $\rho=e^{-a_0\f}$, while the spirals near the circle are comparable with $\rho=(-a_0)^{-1/2}\pm e^{-a_0\f}$. All these spiral trajectories are of exponential type and hence of box dimension equal to 1. See Figure \ref{inf_hopf_fig}, left.

(2) For $a_0=0$ the trajectories of (\ref{htinfl1}) are given with 
$$
\rho(\f)=\sqrt{-2\f+C},\ \ C\in\eR,\ \f\leq\frac{C}{2}.
$$
and infinity is a weak focus with $\dim_B\Gamma=4/3$ where by $\Gamma$ we denote a part of the trajectory near infinity. See Figure \ref{inf_hopf_fig}, middle.

(3) For $a_0>0$ the trajectories of (\ref{htinfl1}) are given with 
$$
\rho(\f)=\frac{1}{a_0}\sqrt{-a_0+a_0^2Ce^{-2a_0\f}},\ \ C>0,\ \f\leq\frac{\ln(a_0C)}{2a_0}.
$$
Infinity is a strong focus and all the trajectories near infinity are comparable with the spiral $\rho=e^{-a_0\f}$ of exponential type, and hence have box dimension equal to $1$. See Figure \ref{inf_hopf_fig}, right.

\begin{figure}[t]
\begin{center}
\includegraphics[height=4cm]{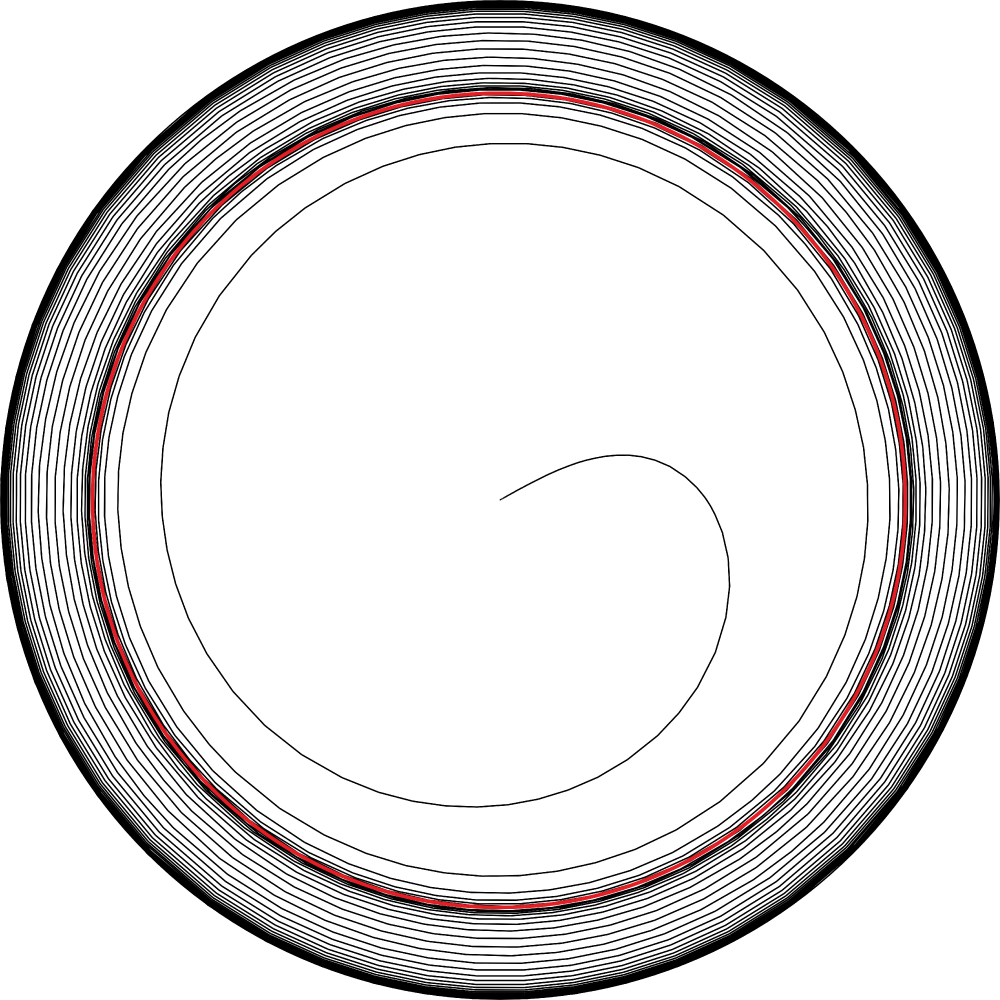}
\includegraphics[height=4cm]{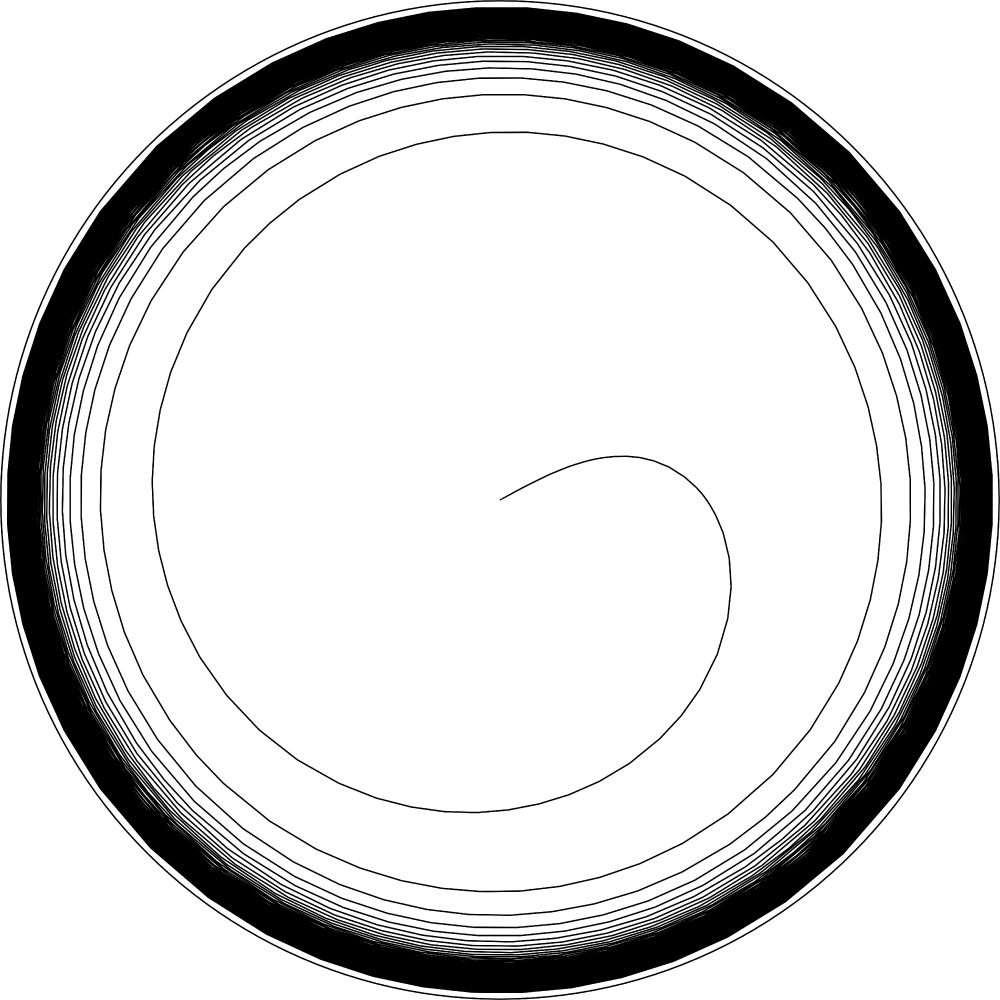}
\includegraphics[height=4cm]{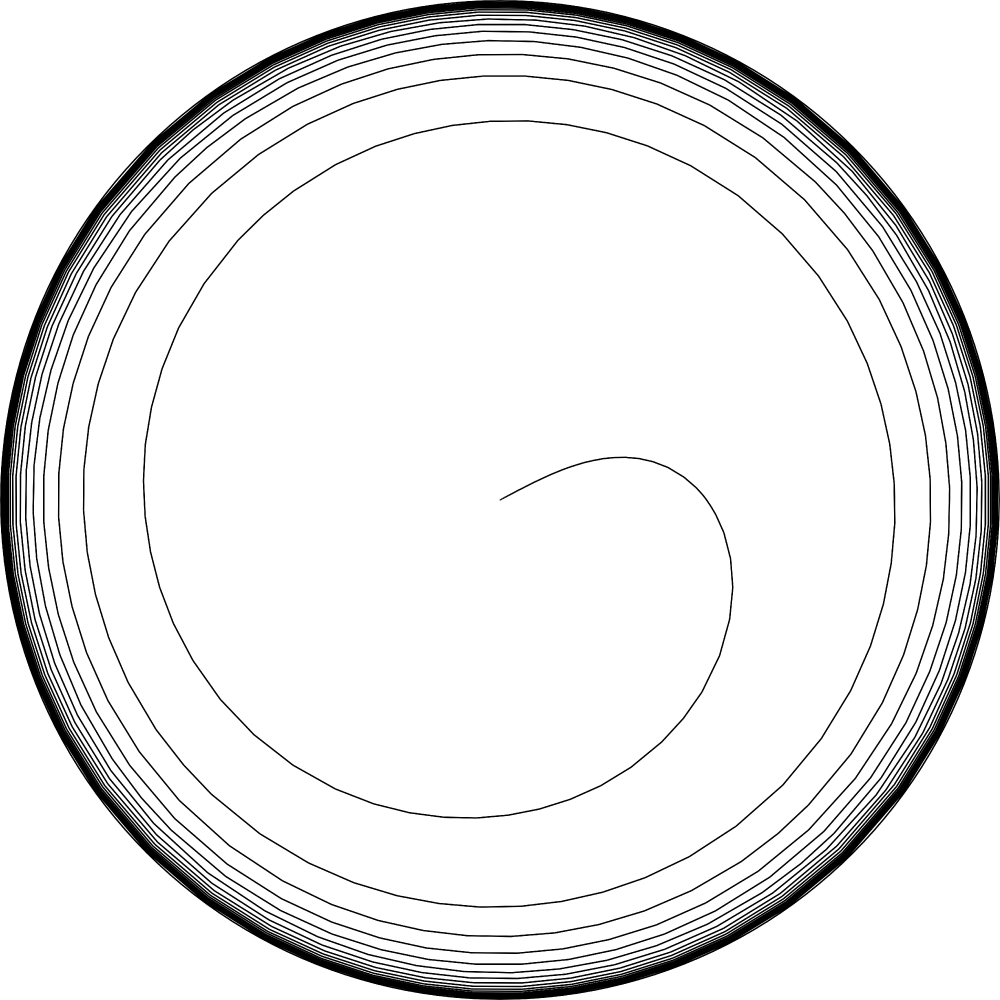}
\end{center}
\caption{Trajectories of the system (\ref{htinfl1}) drawn on the Poincar\' e disc. For $a_0=0$ (middle) the box dimension around the equator would be 1+$\frac{1}{2}\frac{4}{3}=5/3$ but after projecting the half sphere onto the disc the box dimension is reduced to $\frac{2}{3-5/3}=\frac{3}{2}$.}
\label{inf_hopf_fig_poinc}
\end{figure}

Let us now consider the case $l=2$ in (\ref{htinf}):
\bgeq\label{takinf}
\begin{aligned}
\dot \rho&=-\rho(\rho^{-4}+a_0+a_1\rho^{-2})\\
\dot \f&=1.
\end{aligned}
\endeq
Let us fix the value $a_1=-2$ and consider $a_0$ as a bifurcation parameter. Since it is clearer to see what is happening, the phase portrets will be drawn on the Poincar\' e disc.

\begin{figure}[h]
\begin{center}
\includegraphics[height=4cm]{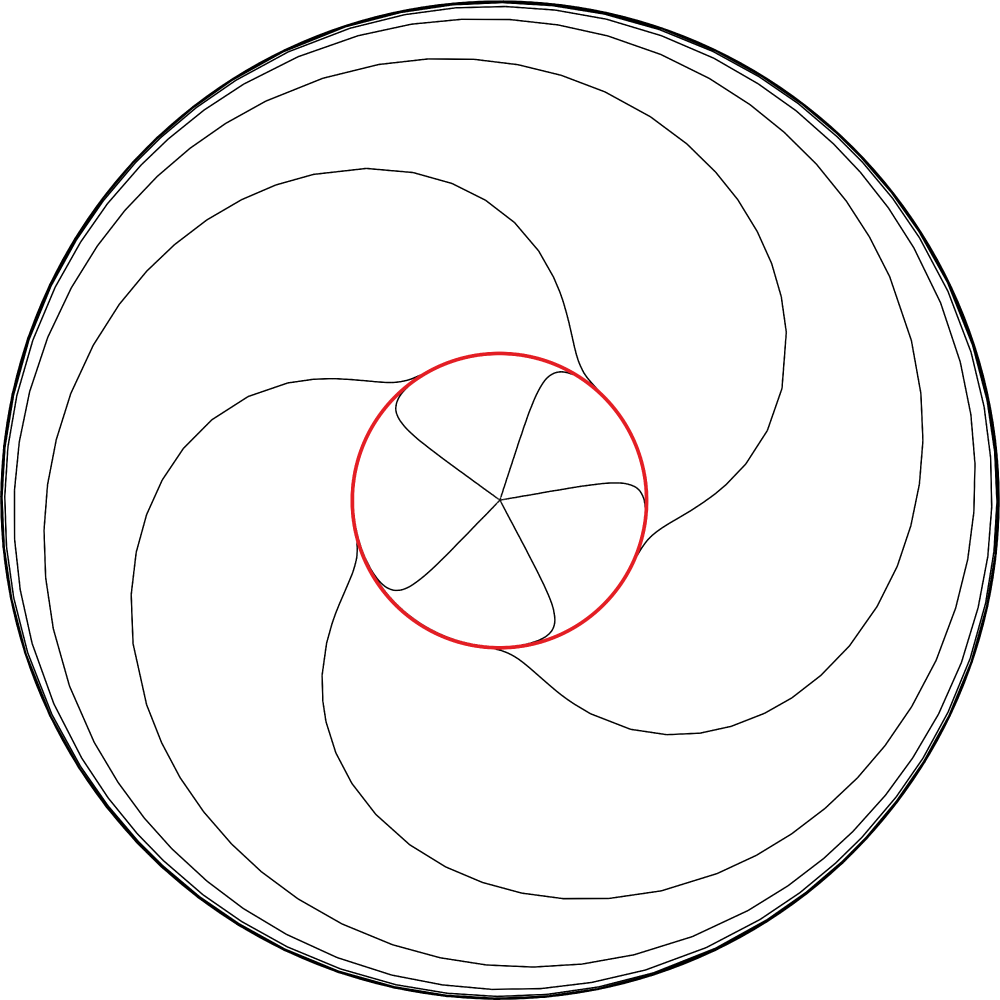}
\includegraphics[height=4cm]{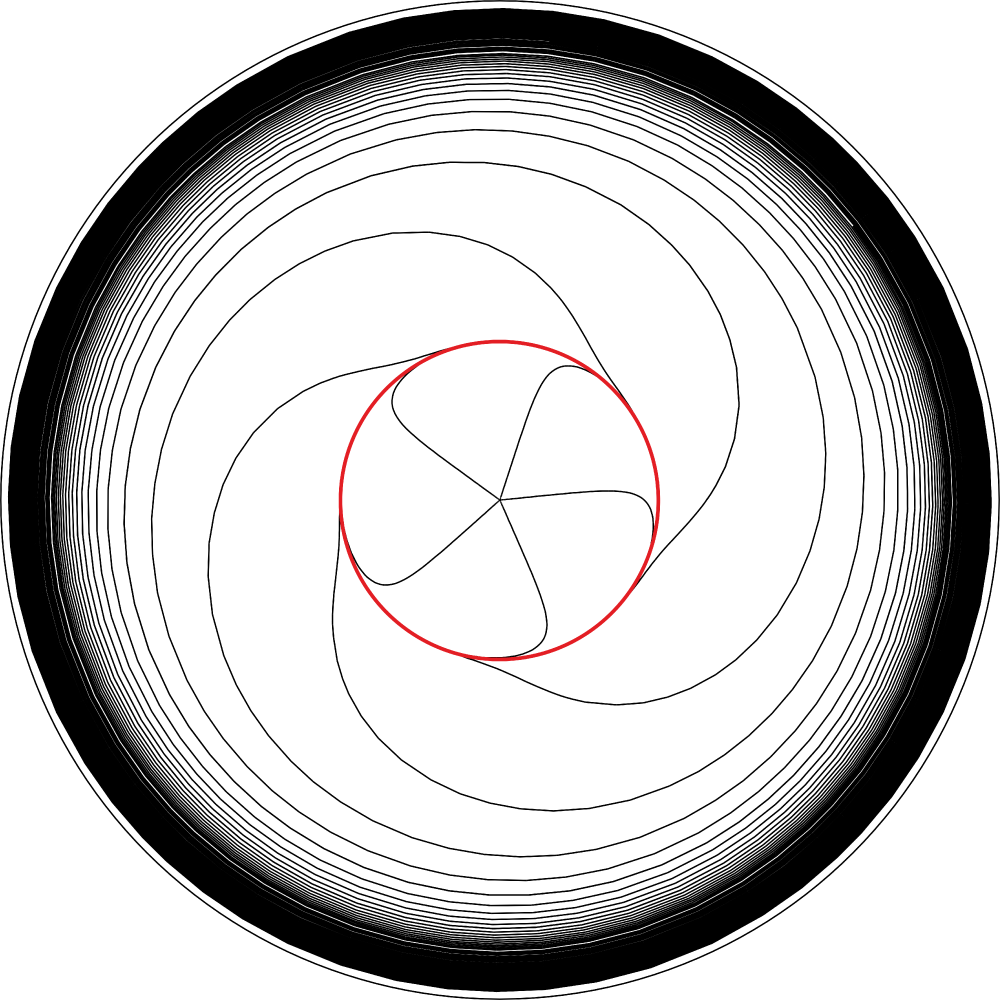}
\includegraphics[height=4cm]{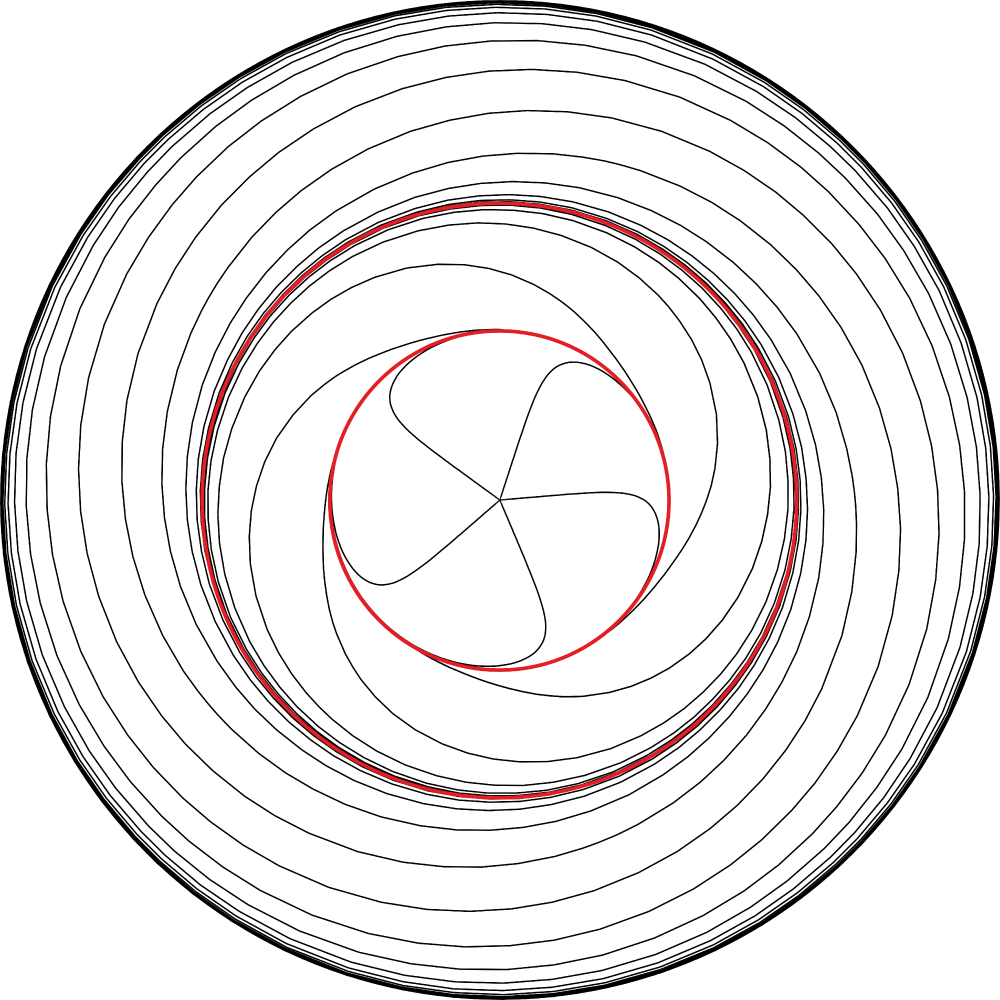}
\end{center}
\caption{}
\label{inf_hopf_tak_fig1}
\end{figure}

\begin{figure}[h]
\begin{center}
\includegraphics[height=4cm]{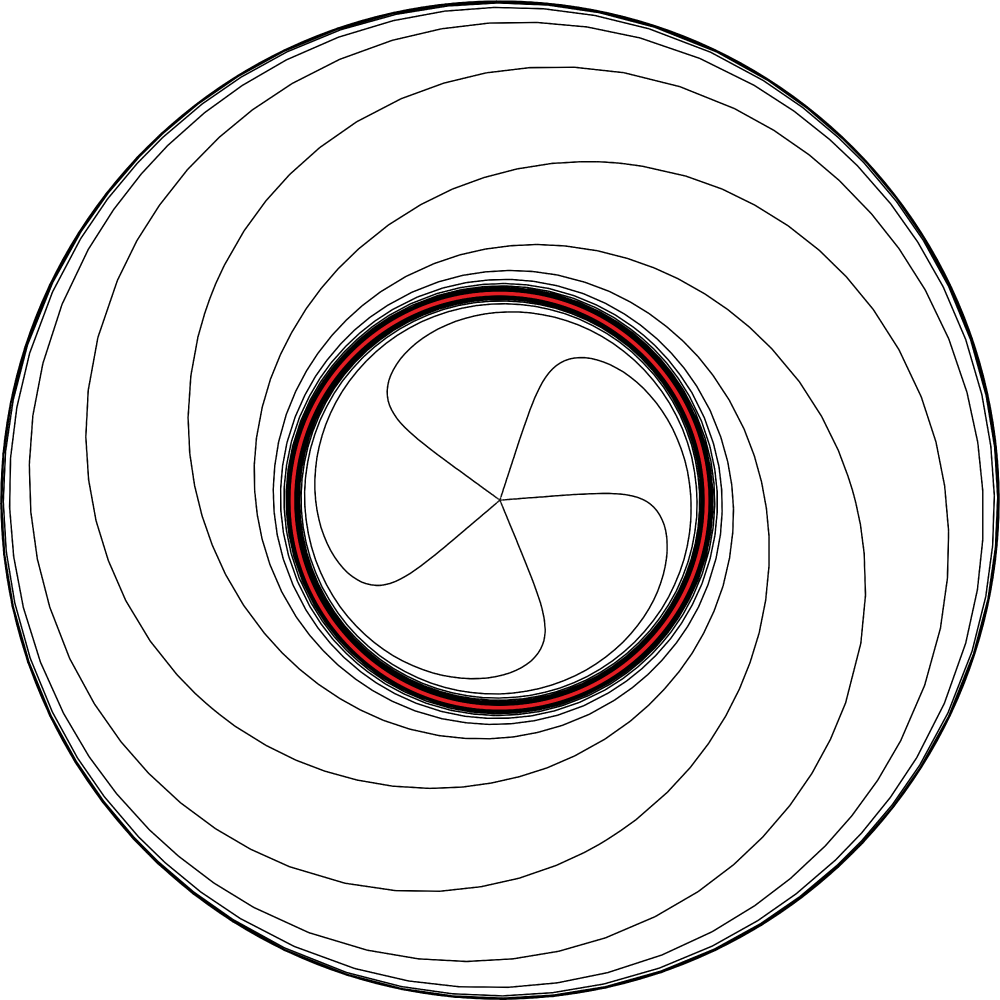}
\includegraphics[height=4cm]{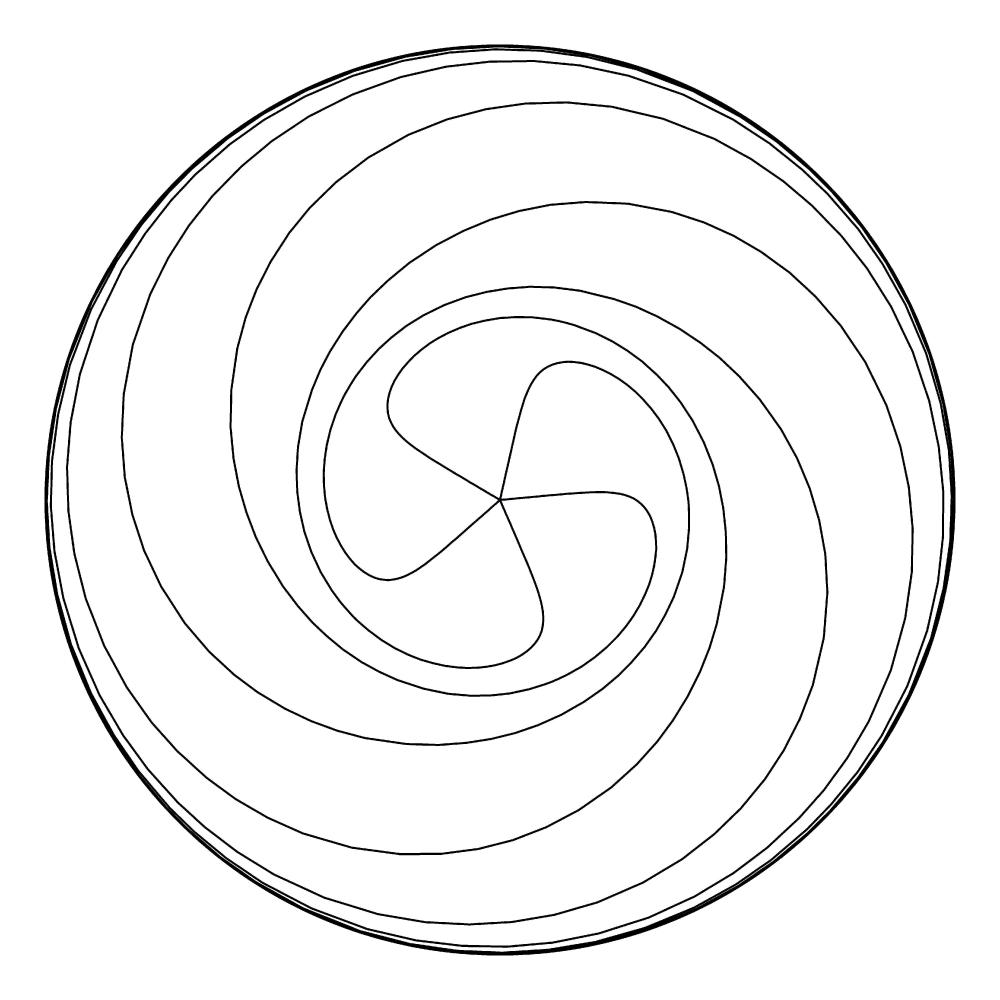}
\end{center}
\caption{}
\label{inf_hopf_tak_fig2}
\end{figure}

(a) When $a_0<0$ all box dimensions are equal to $1$ because all the trajectories are of exponential type, see Figure \ref{inf_hopf_tak_fig1}, left.

(b) For $a_0=0$ we have a weak focus at infinity and any part of a trajectory $\Gamma$ near infinity has box dimension equal to $d=4/3$ (power case), whereas the part near the limit cycle $r=1/\sqrt{2}$ has box dimension equal to $1$ (exponential case), see Figure \ref{inf_hopf_tak_fig1}, middle. Actually, because the process of projecting onto the Poincar\' e disc affects the box dimension (see Remark \ref{dim_trans}), the trajectories on the figure near the equator have box dimension equal to $\frac{4}{4-d}=\frac{3}{2}$.

(c) For $a_0\in(0,1)$ we have two limit cycles of multiplicity one, and all box dimensions are equal to $1$ (exponential case), see Figure \ref{inf_hopf_tak_fig1}, right.

(d) For $a_0=1$ we have a limit cycle $r=1$ of multiplicity two, and all trajectories near the limit cycle (either inside or outside) have box dimension equal to $3/2$ (power case), see Figure \ref{inf_hopf_tak_fig2}, left. On the other hand, trajectories near the equator have box dimension equal to one (exponential case).

(e) For $a_0>1$ box dimensions of all trajectories are equal to one (exponential case), see Figure \ref{inf_hopf_tak_fig2}, right.

\end{document}